\documentclass[11pt]{amsart}

\usepackage{geometry}
\usepackage{amsfonts, adjustbox, amssymb, amscd}
\numberwithin{equation}{section}

\usepackage{bm}
\usepackage{verbatim}
\usepackage{mathrsfs}
\usepackage{graphicx}
\usepackage{tikz-cd}
\usepackage{subcaption}
\usepackage{listings}
\usepackage{subfiles}
\usepackage[toc,page]{appendix}
\usepackage{mathtools}
\usepackage{comment}
\usepackage{enumerate}
\usepackage{enumitem}
\usepackage[all]{xy}

\usepackage{graphicx}
\usepackage{appendix}
\usepackage{hyperref}
\hypersetup{
    colorlinks=true,
    citecolor=red,
    linkcolor=blue,
    filecolor=magenta,      
    urlcolor=red,
}
\newcommand{\bA}{\mathbb{A}}

\newcommand{\cO}{\mathcal{O}}

\newcommand{\bb}{\bm{b}}
\newcommand{\Mm}{{\bf{M}}}

\newcommand{\Qq}{\mathbb{Q}}

\newcommand{\Rr}{\mathbb{R}}

\newcommand{\vol}{\operatorname{vol}}
\newcommand{\Vol}{\operatorname{Vol}}
\newcommand{\nvol}{\widehat{\operatorname{vol}}}
\newcommand{\Center}{\operatorname{center}}

\newcommand{\mld}{{\operatorname{mld}}}
\newcommand{\smooth}{{\operatorname{smooth}}}

\newcommand{\loc}{\operatorname{loc }}

\newcommand{\lct}{\operatorname{lct}}

\newcommand{\Supp}{\operatorname{Supp}}

\newcommand{\mult}{\operatorname{mult}}

\newcommand{\ord}{\mathrm{ord}}

\newcommand{\Ii}{\Gamma}

\newcommand{\Ll}{\mathcal{L}}

\newcommand\MLD{{\rm{MLD}}}


\newcounter{parentnumber}

\newtheorem{thm}{Theorem}[section]
\newtheorem{conj}[thm]{Conjecture}
\newtheorem{cor}[thm]{Corollary}
\newtheorem{lem}[thm]{Lemma}
\newtheorem{prop}[thm]{Proposition}

\theoremstyle{definition}
\newtheorem{defn}[thm]{Definition}

\theoremstyle{definition}
\newtheorem{rem}[thm]{Remark}

\newtheorem{theorem}{Theorem}[section]

\theoremstyle{definition}

\newtheorem{remark}[theorem]{Remark}

\begin{document}

\title{ACC for local volumes}
\author{Jingjun Han, Jihao Liu, and Lu Qi}

\date{\today}

\begin{abstract}
%Xu and Zhuang recently proved the finite coefficient case of the ACC conjecture for local volumes. In this paper we prove the conjecture in full generality.
We prove the ACC conjecture for local volumes. Moreover, when the local volume is bounded away from zero, we prove Shokurov's ACC conjecture for minimal log discrepancies.
\end{abstract}

\address{Shanghai Center for Mathematical Sciences, Fudan University, Jiangwan Campus, Shanghai, 200438, China}
\email{hanjingjun@fudan.edu.cn}

\address{Department of Mathematics, Northwestern University, 2033 Sheridan Road, Evanston, IL 60208, USA}
\email{jliu@northwestern.edu}

\address{School of Mathematical Sciences, East China Normal University, Shanghai 200241, China}
\email{lqi@math.ecnu.edu.cn}

\maketitle

\pagestyle{myheadings}\markboth{\hfill J. Han, J. Liu, and L. Qi \hfill}{\hfill ACC for local volumes\hfill}

\tableofcontents

\section{Introduction}\label{sec:Introduction}
We work over the field of complex numbers \(\mathbb{C}\).

In recent years, there has been significant progress in the algebro-geometric study of K-stability on both global and local scales. Notably, the Yau-Tian-Donaldson conjecture for (potentially singular) log Fano pairs \cite{LXZ22} and the stable degeneration conjecture for klt singularities \cite{XZ22} have been resolved. For further developments, we refer readers to \cite{Blu18, LX18, BX19, LX19, ABHLX20, Jia20, Xu20, BHLLX21, CP21, LWX21, XZ21} and the references therein.

The concept of the local volume of a klt singularity, $\nvol$, introduced by C. Li \cite{Li18}, is central to the theory of local K-stability. For instance, the stable degeneration conjecture \cite[Theorem 1.2]{XZ22} asserts that the \(\nvol\)-minimizer has a finitely generated graded algebra, which degenerates any klt singularity \(x \in (X,B)\) to a K-semistable log Fano cone singularity. Inspired by the K-stability theory and ODP Gap Conjecture \cite[Conjecture 5.5]{SS17} in differntial geometry, the set of local volumes is expected to satisfy the ascending chain condition (ACC).

\begin{conj}[ACC conjecture for local volumes]\label{conj: acc local volume}
    Let $d$ be a positive integer and $\Ii\subset [0,1]$ a set. Consider the set of local volumes
$$\Vol_{d,\Ii}^{\loc}:=\left\{\nvol(X\ni x,B)\Bigm| \dim X=d,B\in\Ii,x\text{ is a closed point}\right\}.$$
\begin{enumerate}
   \item If $\Ii$ is a finite set, then the only accumulation point of $\Vol_{d,\Ii}^{\loc}$ is $0$.
   \item If $\Ii$ satisfies the descending chain condition (DCC), then $\Vol_{d,\Ii}^{\loc}$ satisfies the ACC.
\end{enumerate}
\end{conj}
 Conjecture \ref{conj: acc local volume}(1) was formulated in \cite[Question 6.12]{LLX20} whose embryonic form can be found in \cite[Question 4.3]{LX19}. The full version of Conjecture \ref{conj: acc local volume} was formulated in \cite[Conjecture 8.4]{HLS19} and \cite[Conjecture 1.1]{HLQ23} as a natural extension of Conjecture \ref{conj: acc local volume}(1). %Conjecture \ref{conj: acc local volume} is related to the ODP Gap Conjecture \cite[Conjecture 5.5]{SS17}.

Very recently, Xu and Zhuang \cite[Theorem 1.2]{XZ24} proved Conjecture \ref{conj: acc local volume}(1). The goal of this paper is to prove Conjecture \ref{conj: acc local volume}(2) hence confirm Conjecture \ref{conj: acc local volume} in full generality:

%Conjecture \ref{conj: acc local volume}(2): as a consequence of their proof of the local boundedness conjecture for singularities with finite coefficients \cite[Theorem 1.1]{XZ24}

\begin{thm}\label{thm: ACC local volume}
   Conjecture \ref{conj: acc local volume}(2) holds. Hence, Conjecture \ref{conj: acc local volume} holds.
\end{thm}

Theorem \ref{thm: ACC local volume} can be regarded as a local analog of \cite[Theorem 1.3]{HMX14}, which proved that the set of (global) volumes of log general type pairs with DCC coefficients also satisfies the DCC. Special cases of Theorem \ref{thm: ACC local volume} have been previously established, such as when \(X\) is bounded and \(\mathbb{Q}\)-factorial \cite[Theorem 1.2]{HLQ23}, or when \(\dim X \leq 3\) \cite{Zhu23}.

The proof of Theorem \ref{thm: ACC local volume} relies on \cite[Theorem 1.2]{XZ24}, but it is not a straightforward application. One of the crucial steps in the proof involves establishing the ACC for minimal log discrepancies (mlds) when the local volumes are bounded away from zero.

\begin{thm}\label{thm: acc mld nv lower bound no l}
    Let $d$ be a positive integer, $\epsilon$ a positive real number, and $\Ii\subset [0,1]$ a DCC set. The set
   	$$ \left\{\mld(X\ni x,B)\Bigm| \dim X=d,	B\in\Ii,\nvol(X\ni x,B)\ge\epsilon \right\}$$
			satisfies the ACC.
\end{thm}

The ACC conjecture for mlds, proposed by Shokurov \cite{Sho88}, remains widely open in dimensions three and higher. This conjecture is closely related to the termination of flips and the existence of minimal models  \cite{Sho04,HL22}, one of the remaining core conjectures in the minimal model program. Notably, in dimension four, the first two authors proved that the ACC conjecture for exceptionally non-canonical (enc) singularities implies the termination of flips \cite[Theorem 1.2]{HL22}. It is conjectured that the local volumes of enc singularities have a positive lower bound (cf. \cite[Conjecture 1.8]{HL22}, \cite[Conjecture 1.6]{HLQ23}). Therefore, we expect Theorem \ref{thm: acc mld nv lower bound no l} to be instrumental in proving the termination of flips in dimension four or higher.

%To authors' knowledgement, Theorem \ref{thm: acc mld nv lower bound no l} is the strongest result on the ACC conjecture in arbitarty dimension. 

Special cases of Theorem \ref{thm: acc mld nv lower bound no l} have been previously established when \(X\) is bounded and \(\Qq\)-factorial \cite[Corollary 1.9]{HLQ23}  or when $(X\ni x,B)$ is exceptional \cite[Theorem 1.2]{HLS19}. For recent progress on the ACC conjecture for mlds over the last five years, we refer readers to \cite{HLS19,Jia21,Kaw21,LX21,HLL22,HL22,LL22,NS22a,NS22b,HLQ23,HL23,Kaw23,CGN24,HL24} and references therein.
%NS22b,,NS24,NS23,

%. Additionally, for an \(\epsilon_1\)-lc singularity \((X \ni x, B)\) admitting a \(\delta\)-plt blow-up for some fixed \(\epsilon_1, \delta > 0\), there exists a fixed \(\epsilon > 0\) such that \(\nvol(X \ni x, B) \geq \epsilon\) \cite[Theorem 1.5]{HLQ23}. Consequently, Theorem \ref{thm: acc mld nv lower bound no l} and  \cite[Theorem 1.5]{HLQ23} imply the ACC for mlds for singularities

\medskip

Another key ingredient in the proof of Theorem \ref{thm: ACC local volume} is a result on the \emph{inversion of stability} of $\Rr$-Cartierness for pairs whose local volumes are bounded away from zero. Roughly speaking, ``inversion of stability" of a property $\mathcal{P}$ indicates that for an $\Rr$-divisor $B$ with coefficients belonging to a finite set $\Ii_0$, if $\mathcal{P}$ holds for some $\Rr$-divisor $B'$ sufficiently close to $B$, then $\mathcal{P}$ also holds for $B$. This concept was implicitly proposed by Birkar and Shokurov \cite[Main Proposition 2.1]{BS10} to study the ACC conjecture for mlds (cf. \cite[Theorem 5.19]{HLL22}, \cite[Proposition 6.3]{HL23}), and it is also crucial for studying the boundedness of complements (e.g. \cite{HLS19,Sho20,CH21}). %In this paper, we essentially use the following version of inversion of stability to prove Theorem \ref{thm: ACC local volume}.
In order to prove Theorem \ref{thm: ACC local volume}, we need to reduce to the case of finite coefficients and thus apply \cite{XZ24}. The key issue is that after perturbing the coefficients of a log pair, the log canonical divisor may fail to be $\Rr$-Cartier. To overcome this issue, we prove the following result, which might be of its own interest.

\begin{thm}\label{thm: inversion stability local volume}
    Let $d$ be a positive integer, $\epsilon$ a positive real number, and $\Ii_0\subset [0,1]$ a finite set. Then there exists a positive real number $\tau$ depending only on $d,\epsilon$, and $\Ii_0$ satisfying the following. Assume that
    \begin{enumerate}
        \item $(X\ni x,B')$ is a klt germ of dimension $d$,
        \item $\nvol(X\ni x,B')\geq\epsilon$,
        \item $B\geq B'$, $||B-B'||<\tau$, and $B\in\Ii_0$.
    \end{enumerate}
    Then $(X\ni x,B)$ is klt near $x$. In particular, $K_X+B$ is $\Rr$-Cartier near $x$.
\end{thm}

In the theorem above, the notation $\|\cdot\|$ denotes the maximal absolute value of the coefficients of an $\Rr$-divisor.

Theorem \ref{thm: inversion stability local volume} confirms a special case of \cite[Conjecture 6.6]{HL23} proposed by the first author. In our proof, we require a stronger but more technical version of Theorem \ref{thm: inversion stability local volume}, detailed in Proposition \ref{prop: inversion of stability nv sequence} below. The proof of Theorem \ref{thm: inversion stability local volume} essentially relies on the ACC for mlds (Theorem \ref{thm: acc mld nv lower bound no l}) and the theory on klt complements.

\medskip

\noindent {\it Applications.} Inspired by \cite[Theorem 1.3]{XZ24}, we prove the following result on an upper bound of $\mld^K$:

\begin{thm}\label{thm: upper bound mldk}
    Let $d$ be a positive integer, $\epsilon$ a positive real number, and $\Ii\subset [0,1]$ a DCC set. Then there exists a positive real number $l$ depending only on $d,\epsilon$ and $\Ii$ satisfying the following. Assume that $(X\ni x,B)$ is a klt germ of dimension $d$, $B\in\Ii$, and $\nvol(X\ni x,B)\geq\epsilon$. Then $\mld^K(X\ni x,B)\leq l$, i.e. there exists a Koll\'ar component $E$ of $(X\ni x,B)$ such that $a(E,X,B)\leq l$.
\end{thm}

When the local volumes are bounded away from zero, we prove the uniform boundedness conjecture for mlds due to the first author (see \cite{HL23,HLL22}). We remark that the proofs of several special cases of this conjecture are intertwined with the ACC conjecture for mlds \cite{HLS19,CH21,HLL22,HL23,HL24}.

\begin{thm}\label{thm: Un upper bound mld nv lower bound}
 Let $d$ be a positive integer, $\epsilon$ a positive real number, and $\Ii\subset [0,1]$ a DCC set. Then there exists a positive integer $l$ depending only on $d,\epsilon,\Ii$ satisfying the following. Assume that $(X\ni x,B)$ is a klt $\Qq$-Gorenstein klt germ of dimension $d$, $B\in\Ii$, and $\nvol(X\ni x,B)\geq\epsilon$. Then there exists a prime divisor $E$ over $X\ni x$, such that $a(E,X,B)=\mld(X\ni x,B)$ and $a(E,X,0)\leq l$.
\end{thm}

We also confirm a conjecture of Birkar (see \cite[Conjecture 1.4]{Che24}) provided that the local volumes are bounded away from zero and the ambient variety is $\Qq$-Gorenstein. Note that Birkar's conjecture is equivalent to Shokurov's boundedness of klt complement conjecture \cite[Theorem 1.6]{Che24}.

\begin{thm}\label{thm: Birkar conje lct nv}
  Let $d$ be a positive integer, $\epsilon$ positive real number. Then there exists a positive real number $t$ depending only on $d,\epsilon$ satisfying the following.

Let $(X\ni x,B)$ be a klt germ of dimension $d$ such that $\nvol(X\ni x,B)\ge \epsilon$. Then there exists a Cartier divisor $D>0$ on $X$ near $x$, such that $(X\ni x,B+t D)$ is lc.  
\end{thm}

%\han{use $D_2$ or $G$?}

We note that by the lower semi-continuity of local volumes (\cite[Theorem 1]{BL21} and \cite[Theorem 6.13]{LLX20}), $\nvol(X\ni x,B)\ge \epsilon$ implies that $(X,B)$ is $\frac{\epsilon}{d^d}$-lc near $x$. When the coefficients of $B$ are $\ge\epsilon$, by Theorem \ref{thm:existence of delta-plt blow-up}, $(X\ni x,B)$ admits a $\delta$-plt-blow-up, so we get \cite[Theorem 1.14]{Che24}.

%\begin{rem}Z. Zhuang suggested us that, by (a slightly stronger version of) \cite[Theorem 1.3]{XZ24} on the upper bound of $\mld^K$, one may apply \cite[Theorem 4.8]{Zhu24} and deduce that $(X\ni x,B)$ has a $\delta$-plt blow-up for some positive real number $\delta$ depending only on $\dim X$, the lower bound of non-zero coefficients of $B$, and the lower bound of $\nvol(X\ni x,B)$. This, together with BAB, could possibly lead to simpler proofs of some of our main theorems and also other interesting results. See Section \ref{sec: further discusssion} for details.\end{rem}

\medskip

\noindent\textbf{Acknowledgement}. The authors would like to thank Chen Jiang, Fanjun Meng, Chenyang Xu, and Ziquan Zhuang for useful discussions, and Harold Blum and Yuchen Liu for helpful comments. The first author is affiliated with LMNS at Fudan University. The first author is supported by NSFC for Excellent Young Scientists (\#12322102), and the National Key Research and Development Program of China (\#2023YFA1010600, \#2020YFA0713200). The third author was partially supported by the NSF FRG grant DMS-2139613 and the NSF grant DMS-2201349 of Chenyang Xu. He would also like to thank the Isaac Newton Institute for Mathematical Sciences, Cambridge, for support and hospitality during the programme New equivariant methods in algebraic and differential geometry, where part of the work on this paper was undertaken. This work was partially supported by EPSRC grant EP/R014604/1.

The first author learned Conjecture \ref{conj: acc local volume}(1) from Yuchen Liu when he visited Yale University in October 2018, and he would like to thank their hospitality. 

The second author would like to thank Yuchen Liu for valuable discussions and constant support on this project and related K-stability questions. Although he tried to work out a K-stability problem with Yuchen at Northwestern University for almost three years without success, he finally managed to solve one in the last week when he was still technically affiliated with Northwestern and with Yuchen as the research mentor.

\section{Preliminaries}

We will freely use results on MMP in \cite{KM98,BCHM10}. We follow the definitions in \cite{CH21,XZ24,Zhu24} although the notation may have minor differences.

\begin{defn}
A \emph{pair} $(X/Z\ni z,B)$ consists of a contraction $X\rightarrow Z$ between normal quasi-projective varieties, a (not necessarily closed) point $z\in Z$, and an $\Rr$-divisor $B\geq 0$ such that $K_X+B$ is $\Rr$-Cartier over a neighborhood of $z$.

A divisor $E$ is called \emph{over $X/Z\ni z$} if $\pi(\Center_XE)=\overline{z}$. The \emph{log discrepancy} of a divisor $E$ over $X/Z\ni z$ is $1+\mult_E(K_Y+f^{-1}_*B-f^*(K_X+B))$ where $f: Y\rightarrow X$ is a birational morphism (equivalently, any birational morphism) such that $E$ is on $Y$, and is denoted by $a(E,X,B)$. The \emph{minimal log discrepancy} (mld) of $(X/Z\ni z,B)$ is
$$\mld(X/Z\ni z,B):=\inf_{E}\{a(E,X,B)\mid E\text{ is over }X/Z\ni z\}.$$
It is well-known that if $\mld(X/Z\ni z,B)\geq 0$ then the infimum in the above definition is a minimum (e.g. \cite[Lemma 3.5]{CH21}). We say that $(X/Z\ni z,B)$ is \emph{lc} (resp. $a$-\emph{lc}) if $\mld(X/Z\ni z,B)\geq 0$ (resp. $\geq a$). We say that $(X/Z\ni z,B)$ is \emph{$(a,\mathbb R)$-complementary} if $(X/Z\ni z,B+G)$ is $a$-lc and $K_X+B+G\sim_{\mathbb R}0$ over a neighborhood of $z$ for some $G\geq 0$. Such $(X/Z\ni z,B+G)$ is called an  \emph{$(a,\mathbb R)$-complement} of $(X/Z\ni z,B)$. 

For any $\Rr$-divisor $D\geq 0$ on $X$ that is $\Rr$-Cartier over a neighborhood of $z$, we denote by $\mult_ED:=\mult_E(f^*D-f^{-1}_*D)$ for any prime divisor $E$ over $X/Z\ni z$. The \emph{$a$-lc threshold} of $D$ with respect to $(X/Z\ni z,B)$ is
$$a\text{-}\lct(X/Z\ni z,B;D):=\sup\{t\geq 0\mid (X/Z\ni z,B+tD)\text{ is }a\text{-lc}\}.$$
The \emph{lc threshold} of $D$ with respect to $(X/Z\ni z,B)$ is $0\text{-}\lct(X/Z\ni z,B;D)$ and is also denoted by $\lct(X/Z\ni z,B;D)$.

If $X\rightarrow Z$ is the identity morphism then we may drop $Z$. A pair $(X\ni x,B)$ is called \emph{klt} if $a(E,X,B)>0$ for any $E$ over $X$ such that $\bar x\subset\Center_XE$. A pair $(X\ni x,B)$ is called a \emph{germ} if $x$ is a closed point. If $(X\ni x,B)$ is a (klt, lc, $a$-lc) pair for any $x\in X$ then we say that $(X,B)$ is a (klt, lc, $a$-lc) pair. A pair $(X,B)$ is called \emph{plt} (resp. $a$-plt) if $a(E,X,B)>0$ (resp. $>a$) for any prime divisor $E$ that is exceptional over $X$.

Let $(X,B)$ be a klt pair and $x$ a closed point. A \emph{plt blow-up} of $(X\ni x,B)$ is a blow-up $f: Y\rightarrow X$ with the exceptional divisor $E$ over $X\ni x$, such that $(Y,f^{-1}_*B+E)$ is plt near $E$ and $-E$ is ample$/X$, and $E$ is called a \emph{Koll\'ar component} of $(X\ni x,B)$. Plt blow-up always exists for klt pairs (cf. \cite[Lemma 1]{Xu14}). We denote by
$$\mld^{K}(X\ni x, B):=\min\{ a(E,X,B)\mid  E \text{ is a Koll\'ar component of }(X\ni x,B)\}.$$
The \emph{normalized volume} of $E$ with respect to $(X\ni x,B)$ is
$$\nvol(E,X,B):=a(E,X,B)\cdot\vol(-(K_Y+B_Y+E)|_E).$$
%By \cite[Theorem 1.3]{LX20}, t
The \emph{local volume} of $(X\ni x,B)$ is
$$\nvol(X\ni x,B):=\inf_{E}\nvol(E,X,B)$$
where $E$ runs through all Koll\'ar components of $(X\ni x,B)$. If $(X\ni x,B)$ is not klt near $x$, then we define $\nvol(X\ni x,B):=0$. So when $\nvol(X\ni x,B)>0$, $(X\ni x,B)$ is automatically klt near $x$.

For any klt pair $(X\ni\eta,B)$, the \emph{local volume} of $(X\ni\eta,B)$ is $\nvol(X\ni\eta,B):=\nvol(X\ni x,B)$ where $x$ is a general closed point of $\overline{\eta}$ (cf. \cite[Definition 2.8]{Zhu24}), which is well-defined as $x\rightarrow\nvol(X\ni x,B)$ is constructible \cite[Theorem 1.3]{Xu20}.
\end{defn}

\begin{remark}
    By \cite[Theorem 1.3]{LX20}, the above definition for $\nvol$ is equivalent to the original and usual definition \cite{Li18}:
    \begin{equation}\label{eqn :nv def}
        \nvol(X\ni x,B)\coloneqq \inf_{v\in \mathrm{Val}_{X,x}} a(v,X,B)^d\cdot \vol_X(v).
    \end{equation}
    Here $\mathrm{Val}_{X,x}$ is the set of real valuations $v$ of $K(X)$ centered at $x\in X$, $a(v,X,B)$ is the log discrepancy of $v$ and $\vol_{X}(v)$ is the \emph{volume} of $v$ (computed with respect to $X$); for more details, see for example \cite{LX20}.

   We will not use the original definition, except for the proof of Theorem \ref{thm: Birkar conje lct nv}.
\end{remark}
% By \cite[Theorem 1.3]{LX20}, the definition we adopted is equivalent to \eqref{eqn :nv def}. 

\begin{defn}
Let $\Ii\subset\mathbb R$ be a set. For any $\Rr$-divisor $B=\sum_{i=1}^m b_iB_i$ on a normal variety $X$, where $B_i$ are irreducible components of $B$, we write $B\in\Ii$ if $b_i\in\Ii$ for each $i$. We define $||B||:=\max_{1\le i\le m}\{|b_i|\}$. If $B\in\Ii$, then for any function $p: \Ii\to \Rr$, we define $p(B):=\sum_{i=1}^m p(b_i)B_i$.

A function $p:\Ii\rightarrow\mathbb R$ is called a \emph{projection} if $p\circ p(\gamma)=p(\gamma)$ for any $\gamma\in\Ii$.
\end{defn}

We recall the following results from \cite{Zhu24}. Note that \cite{Zhu24} mainly considers structures with rational coefficients, while we need to consider structures with real coefficients. Nevertheless, these results still hold and can be proven either via the same lines of the proofs, or via perturbation. See \cite[Remark 3.9]{Zhu24} or \cite[Lemma 2.18]{Zhu23} for explanations.

\begin{lem}[{\cite[Lemma 2.11]{Zhu24}}]\label{lem: zhu24 2.11 real}%, {\cite[Lemma 2.16]{HLQ23}}]
Let $d$ be a positive integer and $t$ a positive real number. Let $(X\ni x,B)$ be an lc germ of dimension $d$, and $D$ an $\Rr$-Cartier $\Rr$-divisor on $X$, such that $(X\ni x,B+(1+t)D)$ is lc. Then
$$\nvol(X\ni x,B+D)\geq\left(\frac{t}{1+t}\right)^{d}\nvol(X\ni x,B).$$
\end{lem}

\begin{lem}[{\cite[Lemma 2.10]{Zhu24}}]\label{lem: zhu24 2.10 real}
Let $(X,B)$ and $(Y,B_Y)$ be two klt pairs and $f: Y\rightarrow X$ a proper birational morphism such that $K_Y+B_Y\leq f^*(K_X+B)$. Then for any closed points $x\in X$ and $y\in f^{-1}(x)$, 
$$\nvol(Y\ni y,B_Y)\geq\nvol(X\ni x,B).$$
\end{lem}

\begin{thm}[{\cite[Theorem 1.3]{Zhu24}}]\label{thm: zhu24 1.3 real}
Let $d$ be a positive integer. Then there exists a positive real number $c=c(d)$ depending only on $d$ such that
$$\lct(X\ni x,B;B)\geq c\cdot\nvol(X\ni x,B)$$
for any $d$-dimensional klt germ $(X\ni x,B)$ such that $X\ni x$ is $\Qq$-Gorenstein.
\end{thm}

\begin{lem}\label{lem: xz21 1.4 nonclosed point}
    Let $d$ be a positive integer and $\epsilon$ a positive real number. Then there exists a positive integer $I$ depending only on $d$ and $\epsilon$ satisfying the following.

    Let $(X,B)$ be a pair of dimension $d$ and $W$ a proper closed subvariety of $X$, such that $\nvol(X\ni x,B)\geq\epsilon$ for any closed point $x\in W$. Then for any $\Qq$-Cartier Weil divisor $D$ on $X$ and any (not necessarily closed) point $\eta\in W$, $ID$ is Cartier near $\eta$.
\end{lem}
\begin{proof}
    By \cite[Corollary 1.4]{XZ21}, there exists a positive integer $I$ depending only on $d$ and $\epsilon$, such that $ID$ is Cartier near $x$ for any $\Qq$-Cartier Weil divisor $D$ on $X$ and any closed point $x\in W$. Therefore, $ID$ is Cartier near $\eta$ for any $\Qq$-Cartier Weil divisor $D$ on $X$ and any (not necessarily closed) point $\eta\in W$. 
\end{proof}

\begin{cor}\label{cor: nv lower bound imply mld}
Let $d$ be a positive integer and $\epsilon$ a positive real number. Then there exists a positive integer $l$ depending only on $d$ and $\epsilon$ satisfying the following.

Assume that $(X\ni x,B)$ is a klt germ of dimension $d$ such that $\nvol(X\ni x,B)\geq\epsilon$. Then $\mld(X\ni x,B)\leq l$.  
\end{cor}
\begin{proof}
   Let $f: Y\rightarrow X$ be a small $\Qq$-factorialization, $K_Y+B_Y:=f^*(K_X+B)$, and $y\in f^{-1}(x)$ a closed point. By Lemma \ref{lem: zhu24 2.10 real},
   $$\nvol(Y\ni y)\geq\nvol(Y\ni y,B_Y)\geq\nvol(X\ni x,B)\geq\epsilon.$$
   By \cite[Theorem 1.3]{XZ24}, there exists  a positive integer $l$ depending only on $d$ and $\epsilon$, such that $\mld(Y\ni y)\leq\mld^K(Y\ni y)\leq l$. Thus
   $$\mld(X\ni x,B)=\mld(Y/X\ni x,B_Y)\leq\mld(Y/X\ni x)\leq\mld(Y\ni y)\leq l.$$
\end{proof}

\section{Singularities with local volumes bounded away from zero}\label{sec: acc mld}

\subsection{Lower bound of minimal log discrepancies}

\begin{lem}\label{lem: qfact nv to mld lower bound}
     Let $d$ be a positive integer and $\epsilon$ a positive real number. Then there exists a positive real number $\epsilon_0$ depending only on $d$ and $\epsilon$ satisfying the following. Assume that
    \begin{enumerate}
        \item $(X,B)$ is a klt pair of dimension $d$,
        \item the coefficients of $B$ are $\ge \epsilon$,
        \item $W$ is a proper closed subvariety of $X$ such that $X$ is $\Qq$-factorial near $W$, and
        \item for any closed point $x\in W$, $\nvol(X\ni x,B)\geq\epsilon$.
    \end{enumerate}
    Then for any (not necessarily closed) point $\eta\in W$, $\mld(X\ni\eta,B)\geq\epsilon_0$.
\end{lem}
\begin{proof}
By Lemma \ref{lem: xz21 1.4 nonclosed point}, there exists a positive integer $I$ depending only on $d$ and $\epsilon$, such that for any (not necessarily closed) point $\eta\in W$ and any Weil divisor $D$ on $X$, $ID$ is Cartier near $\eta$. If $B=0$ near $\eta$ then $\mld(X\ni\eta,B)\geq\frac{1}{I}$ and we are done, so we may assume that $B\not=0$ near $\eta$.

By Theorem \ref{thm: zhu24 1.3 real}, there exists a positive real number $c$ depending only on $d$ and $\epsilon$ such that $(X\ni x,(1+c)B)$ is lc for any closed point $x\in W$. Thus $(X\ni \eta,(1+c)B)$ is lc for any (not necessarily closed) point $\eta\in W$. 

Let $S$ be a component of $B$ which passes through $\eta$. Since $(X\ni \eta,(1+c)B)$ is lc and $\mult_SB\geq\epsilon$,  $(X\ni\eta,B+c\epsilon S)$ is lc. Let $E$ be a divisor over $X\ni\eta$ such that $a(E,X,B)=\mld(X\ni \eta,B)$. Since $IS$ is Cartier near $x$, $\mult_ES\geq\frac{1}{I}$.
Then
$$0\leq a(E,X,B+c\gamma_0S)=a(E,X,B)-c\gamma_0\mult_ES\leq\mld(X\ni\eta,B)-\frac{c\epsilon}{I},$$
so we may take $\epsilon_0=\frac{c\epsilon}{I}$.
\end{proof}

\begin{lem}\label{lem: closed point nv lower bound implies mld lower bound}
      Let $d$ be a positive integer and $\epsilon$ a positive real number. Then there exists a positive real number $\epsilon_0$ depending only on $d$ and $\epsilon$ satisfying the following. 

      Assume that $(X\ni x,B)$ is a klt germ of dimension $d$ such that the coefficients of $B$ are $\ge \epsilon$ and $\nvol(X\ni x,B)\geq\epsilon$. Then $\mld(X\ni x,B)\geq\epsilon_0$.
\end{lem}
\begin{proof}
Let $f: Y\rightarrow X$ be a small $\Qq$-factorialization near $x$ and let $K_Y+B_Y:=f^*(K_X+B)$. Then $(Y,B_Y)$ is klt, the coefficients of $B_Y$ are $\ge \epsilon$, and $Y$ is $\Qq$-factorial near $f^{-1}(x)$. By Lemma \ref{lem: zhu24 2.10 real}, for any closed point $y\in f^{-1}(x)$, $\nvol(Y\ni y,B_Y)\geq\nvol(X\ni x,B)\geq\epsilon$. By Lemma \ref{lem: qfact nv to mld lower bound}, for any (not necessarily closed) point $\eta\in f^{-1}(x)$, $\mld(Y\ni\eta,B_Y)\geq\epsilon_0$. Thus
$$\mld(X\ni x,B)=\inf_{\eta\in f^{-1}(x)}\mld(Y\ni\eta,B_Y)\geq\epsilon_0.$$
\end{proof}

\subsection{ACC for minimal log discrepancies}

\begin{lem}\label{lem: acc mld nv lower bound q factorial}
    Let $d$ be a positive integer, $\epsilon,l$ two positive real numbers, and $\Ii\subset [0,1]$ a DCC set. Then there exists an ACC set $\MLD(d,\Ii,\epsilon,l)\subset [0,l]$ depending only on $d,\epsilon,l$ and $\Ii$ satisfying the following. Assume that $X$ is a normal variety, $\eta\in X$ is a (not necessarily closed) point, and $B\geq 0$ is an $\Rr$-divisor, such that
    \begin{enumerate}
        \item $(X,B)$ is $\Qq$-factorial klt near $W:=\overline{\eta}$,
        \item $\dim X=d$ and $B\in\Ii$,
        \item $\nvol(X\ni x,B)\geq\epsilon$ for any closed point $x\in W$, and
        \item $\mld(X\ni\eta,B)\leq l$.
    \end{enumerate}
    Then $\mld(X\ni \eta,B)\in\MLD(d,\Ii,\epsilon,l)$.
\end{lem}
\begin{proof}
    Suppose the lemma does not hold. Then there exists a sequence $(X_i\ni\eta_i,B_i)$ and $W_i:=\overline{\eta_i}$ satisfying the conditions of $(X\ni\eta,B)$ and $W$ as above, such that $a_i:=\mld(X_i\ni\eta_i,B_i)$ is strictly increasing. We may assume that all components of $B_i$ pass through $\eta_i$. Since $X_i$ is $\Qq$-factorial near $\eta_i$ and $B_i\in\Ii$, by \cite[18.22 Theorem]{Kol+92}, possibly replacing $B_i$ and passing to a subsequence, we may assume that there exists a positive integer $m$ and real numbers $\{b_{i,j}\}_{i\geq 1,1\leq j\leq m}$, such that
    $$B_i=\sum_{j=1}^mb_{i,j}B_{i,j},$$
    where $B_{i,j}$ are the irreducible components of $B_i$, $b_{i,j}\in\Ii$, and $\{b_{i,j}\}_{i=1}^{+\infty}$ is increasing for any fixed $j$. We let $\bar b_j:=\lim_{i\rightarrow+\infty}b_{i,j}$ for each $j$ and let $\bar B_i:=\sum_{j=1}^m \bar b_jB_{i,j}$. 

    Let $c:=c(d)$ be the real number defined in Theorem \ref{thm: zhu24 1.3 real}. Then $(X_i\ni x_i,(1+c\epsilon)B_i)$ is lc for any closed point $x_i\in W_i$, so $(X_i\ni\eta_i,(1+c\epsilon)B_i)$ is lc. Since $B_i\in\Ii$, possibly passing to a subsequence, we may assume that 
    $$(1+c\epsilon)B_i\geq\bar B_i$$
    for each $i$. Then $(X_i\ni\eta_i,\bar B_i)$ is lc for each $i$.  By Lemma \ref{lem: xz21 1.4 nonclosed point}, there exists a positive integer $I$ depending only on $d$ and $\epsilon$, such that $ID$ is Cartier near $\eta_i$ for any Weil divisor $D$ on $X_i$. By \cite[Theorem 1.2]{Nak16}, $\mld(X_i\ni\eta_i,\bar B_i)$ belongs to a discrete set. Since
    $$\mld(X_i\ni\eta_i,\bar B_i)\leq \mld(X_i\ni\eta_i,B_i)\leq l,$$
$\mld(X_i\ni\eta_i,\bar B_i)$ belongs to a finite set. Possibly passing to a subsequence, we may assume that $a:=\mld(X_i\ni\eta_i,\bar B_i)$ is a constant. Since $a_i$ is strictly increasing, $a_i-a$ is strictly increasing. In particular, we may assume that $a_i>a$ for each $i$.

Let $E_i$ be a prime divisor over $X_i\ni\eta_i$ such that $a=a(E_i,X_i,\bar B_i)$. Then $a(E_i,X_i,B_i)\geq a_i$. Thus $\mult_{E_i}(\bar B_i-B_i)\geq a_i-a$, so
$$a\left(E_i,X_i,\bar B_i+\frac{2a}{a_i-a}(\bar B_i-B_i)\right)<0.$$
Thus $$\left(X_i\ni\eta_i,\bar B_i+\frac{2a}{a_i-a}(\bar B_i-B_i)\right)=\left(X_i\ni\eta_i,B_i+\left(1+\frac{2a}{a_i-a}\right)(\bar B_i-B_i)\right)$$ is not lc. However, for any $i\gg 0$, 
$$(1+c\epsilon)B_i\geq B_i+\left(1+\frac{2a}{a_i-a}\right)(\bar B_i-B_i),$$
so $(X_i\ni\eta_i,(1+c\epsilon)B_i)$ is not lc, a contradiction.
\end{proof}

\begin{thm}\label{thm: acc mld nv lower bound}
    Let $d$ be a positive integer, $\epsilon,l$ two positive real numbers, and $\Ii\subset [0,1]$ a DCC set. Assume that 
    \begin{enumerate}
        \item $(X\ni x,B)$ is a klt germ of dimension $d$,
        \item $B\in\Ii$,
        \item $\nvol(X\ni x,B)\geq\epsilon$, and
        \item $\mld(X\ni x,B)\leq l$.
    \end{enumerate}
    Then $\mld(X\ni x,B)\in\MLD(d,\Ii,\epsilon,l)$ where $\MLD(d,\Ii,\epsilon,l)$ is the set as in Lemma \ref{lem: acc mld nv lower bound q factorial}.
\end{thm}
\begin{proof}
Let $f: Y\rightarrow X$ be a small $\Qq$-factorialization near $x$, and $K_Y+B_Y:=f^*(K_X+B)$. Then there exists a (not necessarily closed) point $\eta\in f^{-1}(x)$, such that
$$\mld(Y\ni\eta,B_{Y})=\mld(X\ni x,B)\leq l$$
By our construction, $(Y,B_Y)$ is $\Qq$-factorial klt near $\overline{\eta}$, $\dim Y=d$, and $B_Y\in\Ii$. By Lemma \ref{lem: zhu24 2.10 real}, 
$$\nvol(Y\ni y,B_Y)\geq\nvol(X\ni x,B)\geq\epsilon$$
for any closed point $y\in\overline{\eta}$. By Lemma \ref{lem: acc mld nv lower bound q factorial}, $\mld(X\ni x,B)=\mld(Y\ni\eta,B_{Y})\in\MLD(d,\Ii,\epsilon,l)$.
\end{proof}

\subsection{ACC for \texorpdfstring{$a$}{}-lc thresholds}

\begin{lem}\label{lem: acc alct nv lower bound q factorial}
    Let $d$ be a positive integer, $\epsilon,a$ two positive real numbers, and $\Ii\subset [0,1]$ a DCC set. Then there exists an ACC set $\Ii'$ and a positive real number $s$ depending only on $d,\epsilon,a$ and $\Ii$ satisfying the following. Assume that $X$ is a normal variety, $\eta\in X$ is a (not necessarily closed) point, and $B\geq 0$ is an $\Rr$-divisor, and $S$ a component of $B$, such that
    \begin{enumerate}
        \item $(X,B)$ is $\Qq$-factorial klt near $W:=\overline{\eta}$,
        \item $\dim X=d$ and $B\in\Ii$,
        \item $\nvol(X\ni x,B)\geq\epsilon$ for any closed point $x\in W$, and
        \item $a\leq\mld(X\ni\eta,B)$.
    \end{enumerate}
    Then $a$-$\lct(X\ni\eta,B-(\mult_SB)S;S)\in\Ii'\cup [\mult_SB+s,+\infty)$.
\end{lem}
\begin{proof}
We may assume that $\Ii\not=\emptyset$. Let $\gamma_0:=\min\Ii_{>0}$.

By Theorem \ref{thm: zhu24 1.3 real}, there exists a positive real number $c$ depending only on $d$ and $\epsilon$ such that $(X\ni x,(1+c)B)$ is lc for any closed point $x\in W$. Thus  $(X\ni \eta,(1+c)B)$ is lc, so $(X\ni\eta, B+c\gamma_0S)$ is lc. Let $s:=\frac{1}{2}c\gamma_0$. 

Suppose the lemma does not hold. Then there exists a sequence $(X_i\ni\eta_i,B_i)$, $S_i$ and $W_i:=\overline{\eta_i}$ satisfying the conditions (1)-(4) above, such that, $t_i:=a$-$\lct(X_i\ni\eta_i,B_i-b_iS_i;S_i)$ is strictly increasing, and $t_i<b_i+s$ for any $i$, where $b_i:=\mult_{S_i}B_i$. Let $t:=\lim_{i\rightarrow+\infty}t_i$, then by the ACC for log canonical thresholds \cite[Theorem 1.1]{HMX14}, $(X_i\ni x_i,B_i-b_iS_i+tS_i+\frac{1}{2}c\gamma_0S_i)$ is lc for any closed point $x_i\in W_i$, by Lemma \ref{lem: zhu24 2.11 real}, there exists a positive real number $\epsilon_1$ depending only on $d,\epsilon$ and $\Ii$, such that $\nvol(X_i\ni x_i,B_i-b_iS_i+tS_i)\geq\epsilon_1$. Let $a_i':=\mld(X_i\ni\eta_i,B_i-b_iS_i+tS_i)$. Since $a_i'<a$, by Theorem \ref{thm: acc mld nv lower bound}, possibly passing to a subsequence, we may assume that $a_i'$ is decreasing. Let $E_i$ be a prime divisor over $X_i\ni\eta_i$ such that $a(E_i,X_i,B_i-b_i S_i+tS_i)=a_i'$. Then $a(E_i,X_i,B_i-b_i S_i+t_iS_i)\ge a>a_i'$, and so $\mult_{E_i}S_i>\frac{a-a_i'}{t-t_i}$. Thus
\[
    a\left(E_i,X_i,B_i+\left(t_i-b_i+\frac{2a(t-t_i)}{a-a_i'}\right)S_i\right)\leq -a<0,
\]
so $\left(X_i\ni\eta_i,B_i+\left(t_i-b_i+\frac{2a(t-t_i)}{a-a_i'}\right)S_i\right)$ is not lc. This is not possible as $$b_i+c\gamma_0>b_i+\frac{1}{2}c\gamma_0+\frac{2a(t-t_i)}{a-a_i'}>t_i+\frac{2a(t-t_i)}{a-a_i'}$$
for any $i\gg 0$ and $(X_i\ni\eta_i,B_i+c\gamma_0S_i)$ is lc.
\end{proof}

The following lemma is a slightly strengthened version of \cite[Lemma 5.17]{HLS19}.

%\han{$p(0)=0$?}
\begin{lem}\label{lem:hls19 5.17 improved}
Let $M$ be a positive real number, $\Ii\subset [0,M]$ a DCC set, and $\Ii'=\bar\Ii'$ an ACC set. Then for any positive real number $\alpha$, there exist a finite set $\Ii_{\alpha}\subset\bar\Ii$ and a projection $p_{\alpha}:\bar\Ii\rightarrow\Ii_{\alpha}$ satisfying the following:
\begin{enumerate}
    \item $\gamma+\alpha\geq p_{\alpha}(\gamma)\geq\gamma$ for any $\gamma\in\Ii$,
    \item $p_{\alpha}(\gamma')\geq p_{\alpha}(\gamma)$ for any $\gamma',\gamma\in\Ii$ such that $\gamma'\geq\gamma$, 
    \item for any $\beta\in\Ii'$ and $\gamma\in\Ii$, if $\beta\geq\gamma$, then $\beta\geq p_{\alpha}(\gamma)$, and
    \item if $\alpha=n\alpha'$ for some positive integer $n$, then $p_{\alpha}(\gamma)\geq p_{\alpha'}(\gamma)$ for any $\gamma\in\Ii$.
\end{enumerate}
\end{lem}
\begin{proof}
    We may replace $\Ii$ with $\bar\Ii$ and assume that $\Ii=\bar\Ii$. Let $N_{\alpha}:=\lceil\frac{M}{\alpha}\rceil$, $\Ii_{\alpha,0}:=\Ii\cap [0,\frac{1}{N}]$, and $\Ii_{\alpha,k}:=\Ii\cap (k\alpha,(k+1)\alpha]$ for any $1\leq k\leq N_{\alpha}-1$. 

First, for any fixed $\alpha$, we construct $p_{\alpha}$. For any $\gamma\in\Ii$ and $\alpha$, there exists a unique $0\leq k\leq N_{\alpha}-1$ such that $\gamma\in\Ii_{\alpha,k}$. 

If $\gamma\in\Ii$ and $\gamma>\max\{\beta\mid\beta\in\Ii'\}$, then we let $p_{\alpha}(\gamma):=\max\{\beta\mid\beta\in\Ii_{\alpha,k}\}$. It is clear that $p_{\alpha}(\gamma)$ belongs to a finite set and $p_{\alpha}(p_{\alpha}(\gamma))=p_{\alpha}(\gamma)$.

Now for any $\gamma\in\Ii$ such that $\gamma\leq\max\{\beta\mid\beta\in\Ii'\}$, we define
$$f(\gamma):=\min\{\beta\in\Ii''\mid\beta\geq\gamma\}$$
and
$$p_{\alpha}(\gamma):=\max\{\beta\in\Ii_{\alpha,k}\mid\beta\leq f(\gamma),\gamma\in\Ii_{\alpha,k}\}.$$

Now we show that $\Ii_{\alpha}:=\{p_{\alpha}(\gamma)\mid\gamma\in\Ii\}$ satisfies our requirements. We prove (1-4) first. For any $\gamma\in\Ii$, it is clear that 
$$0\leq p_{\alpha}(\gamma)-\gamma\leq\alpha,$$
this implies (1). (2) and (3) follow from our construction. To prove (4), note that for any positive integer $n$ and $\alpha=n\alpha'$, and any $0\leq k'\leq N_{\alpha'}-1$, there exists a unique $0\leq k\leq N_{\alpha}-1$ such that $\Ii_{\alpha',k'}\subset\Ii_{\alpha,k}$. Thus if $\gamma>\max\{\beta\mid\beta\in\Ii'\}$, we have
    $$p_{\alpha'}(\gamma)=\max\{\beta\mid\beta\in\Ii_{\alpha',k'}\}\leq\max\{\beta\mid\beta\in\Ii_{\alpha,k}\}=p_{\alpha}(\gamma),$$
and if $\gamma\leq\max\{\beta\mid\beta\in\Ii'\}$, we have
    $$p_{\alpha'}(\gamma):=\max\{\beta\in\Ii_{\alpha',k'}\mid\beta\leq f(\gamma),\gamma\in\Ii_{\alpha',k'}\}\leq\max\{\beta\in\Ii_{\alpha,k}\mid\beta\leq f(\gamma),\gamma\in\Ii_{\alpha,k}\}=p_{\alpha}(\gamma).$$
    Thus (4) holds. 

Since $f(\gamma)\geq p_{\alpha}(\gamma)\geq\gamma$, we have $f(p_{\alpha}(\gamma))=f(\gamma)$ and $p_{\alpha}(p_{\alpha}(\gamma))=p_{\alpha}(\gamma)$, so $p_{\alpha}$ is a projection. We are left to show that $\Ii_{\alpha}$ is a finite set. 
is a finite set. Since $\Ii_{\alpha}\subset\Ii$, $\Ii_{\alpha}$ satisfies the DCC, so we only need to show that $\Ii_{\alpha}$ satisfies the ACC. Suppose that there exists a strictly increasing sequence $p_{\alpha}(\gamma_1)<p_{\alpha}(\gamma_2)<\dots$, where $\gamma_i\in\Ii$. Since $f(\gamma_i)\in\Ii'$, possibly passing to a subsequence,
we may assume that $f(\gamma_i)$ is decreasing. Thus $g(\gamma_i)$ is decreasing, and we
get a contradiction. Thus  $\Ii_{\alpha}$ is a finite set and we are done.
\end{proof}

%\han{if mld has an upper bound then easy?}

\begin{prop}\label{prop: alct uniform increasing}
        Let $d$ be a positive integer, $\epsilon,a$ two positive real numbers, and $\Ii\subset [0,1]$ a DCC set such that $\bar\Ii=\Ii$. Then there exists a positive real number $s$ depending only on $d,\epsilon,a$ and $\Ii$, such that for any positive real number $\alpha<s$, there exist a finite set $\Ii_{\alpha}\subset\Ii$ and a projection $p_{\alpha}: \Ii\rightarrow\Ii_{\alpha}$ depending only on $d,\epsilon,\alpha,a$ and $\Ii$ such that 
        \begin{enumerate}
           \item  $\gamma+\alpha\geq p_{\alpha}(\gamma)\geq\gamma$ for any $\gamma\in\Ii$,
           \item $p_{\alpha}(\gamma')\geq p_{\alpha}(\gamma)$ for any $\gamma',\gamma\in\Ii$ such that $\gamma'\geq\gamma$,
            \item if $\alpha=n\alpha'$ for some positive integer $n$, then $p_{\alpha}(\gamma)\geq p_{\alpha'}(\gamma)$ for any $\gamma\in\Ii$.
        \end{enumerate}
and satisfies the following. Assume that $X$ is a normal variety, $\eta\in X$ is a (not necessarily closed) point, and $B\geq 0$ is an $\Rr$-divisor, such that
    \begin{itemize}
        \item $(X,B)$ is $\Qq$-factorial klt near $W:=\overline{\eta}$,
        \item $\dim X=d$ and $B=\sum_{i=1}^m b_iB_i$, where $B_i$ are the irreducible components of $B$ and $b_i\in\Ii$ for each $i$,
        \item $\nvol(X\ni x,B)\geq\epsilon$ for any closed point $x\in W$, and
        \item $a\leq\mld(X\ni\eta,B)$.
    \end{itemize}
    Then
    \begin{enumerate}
        \item[(4)] $(X\ni\eta, p_{\alpha}(B):=\sum_{i=1}^m p_{\alpha}(b_i)B_i)$ is $a$-lc.
    \end{enumerate}
\end{prop}
\begin{proof}
We let $\Ii'$ and $s$ be the set and the real number given in Lemma \ref{lem: acc alct nv lower bound q factorial} which depends only on $d,\epsilon,a$ and $\Ii$. By Lemma \ref{lem:hls19 5.17 improved}, for any positive real number $\alpha$, there exists a finite set $\Ii_{\alpha}\subset\Ii$ and a projection $p_{\alpha}: \Ii\rightarrow\Ii_{\alpha}$ such that (1-3) hold, and also satisfies the following: 
\begin{itemize}
    \item[$(*)$] For any $\beta\in\Ii'$ and $\gamma\in\Ii$, if $\beta\geq\gamma$, then $\beta\geq p_{\alpha}(\gamma)$.
\end{itemize}
We only need to prove (4). In the following we fix $\alpha$ and let $p:=p_{\alpha}$. Suppose (4) does not hold, then there exists $0\leq k\leq m-1$, such that 
$$\left(X\ni\eta,\sum_{i=1}^kp(b_i)B_i+\sum_{j=k+1}^mb_iB_i\right)$$
is $a$-lc, but
$$\left(X\ni\eta,\sum_{i=1}^{k+1}p(b_i)B_i+\sum_{j=k+2}^mb_iB_i\right)$$
is not $a$-lc. Let
$$\mu:=a\text{-}\lct\left(X\ni\eta,\sum_{i=1}^kp(b_i)B_i+\sum_{j=k+1}^mb_iB_i;B_{k+1}\right),$$
then 
$$b_{k+1}\leq\mu<p(b_{k+1})\leq b_{k+1}+\alpha<b_{k+1}+s.$$
Since $b_i,p(b_i)\in\Ii$ for each $i$, and since $\mu<b_{k+1}+s$, we have $\mu\in\Ii'$. This contradicts condition $(*)$.
\end{proof}

\section{Inversion of stability}

The following proposition is crucial to the proof of the main results. The proof uses the idea in \cite[Theorem 5.19]{HLL22}, which is based on the ACC for mlds and the theory of complements \cite{HLS19}.

\begin{prop}\label{prop: inversion of stability nv sequence}
    Let $d$ be a positive integer, $\epsilon,l$ two positive real numbers, and $\Ii\subset [0,1]$ a DCC set such that $\bar\Ii=\Ii$. Assume that 
     \begin{itemize}
        \item $(X_i\ni x_i,B_i)$ is a sequence of klt germ of dimension $d$,
        \item $B_i=\sum_{j=1}^{m_i} b_{i,j}B_{i,j}$, where $B_{i,j}$ are the irreducible components of $B_i$ and $b_{i,j}\in\Ii$,
        \item $x_i\in \Supp B_{i,j}$ for any $i,j$,
        \item $\nvol(X_i\ni x_i,B_i)\geq\epsilon$, and
        \item $\mld(X_i\ni x_i,B_i)\leq l$.
    \end{itemize}
    Then there exist two positive real numbers $a,s$ such that for any positive real number $\alpha<s$, there exists a finite set $\Ii_{\alpha}\subset\Ii$ and a projection $p_{\alpha}:\Ii\rightarrow\Ii_{\alpha}$ satisfying the following. Possibly passing to a subsequence,
    \begin{enumerate}
        \item $\gamma+\alpha\geq p_{\alpha}(\gamma)\geq\gamma$ for any $\gamma\in\Ii$,
        \item $p_{\alpha}(\gamma')\geq p_{\alpha}(\gamma)$ for any $\gamma',\gamma\in\Ii$ such that $\gamma'\geq\gamma$,
        \item if $\alpha=n\alpha'$ for some positive integer $n$, then $p_{\alpha}(\gamma)\geq p_{\alpha'}(\gamma)$ for any $\gamma\in\Ii$,
        \item $a=\lim_{i\rightarrow+\infty}\mld(X_i\ni x_i,B_i)$,
        \item $K_{X_i}+p_{\alpha}(B_i)$ is $\Rr$-Cartier, and $\mld(X_i\ni x_i,p_{\alpha}(B_i))=a$, and
        \item $\lim_{i\rightarrow+\infty}||p_{\alpha}(B_i)-B_i||=0$.
    \end{enumerate}
    %:=\sum_{i=1}^{m_i}p_{\alpha}(b_{i,j})B_{i,j}
\end{prop}
\begin{proof}
Let $a_i:=\mld(X_i\ni x_i,B_i)$. By Theorem \ref{thm: acc mld nv lower bound}, possibly passing to a subsequence, we may assume that $a_i$ is decreasing and let $a:=\lim_{i\rightarrow+\infty}a_i$. By Lemma \ref{lem: closed point nv lower bound implies mld lower bound}, $a>0$. Let $s$ be the positive real number given in Proposition \ref{prop: alct uniform increasing}, and for any positive real number $\alpha<s$, let $\Ii_{\alpha}\subset\Ii$ be the finite set and $p_{\alpha}: \Ii\rightarrow\Ii_{\alpha}$ the projection given in Proposition \ref{prop: alct uniform increasing} depending only on $d,\epsilon,\alpha$ and $a$. (1-4) hold by our construction.

Let $f_i: Y_i\rightarrow X_i$ be a small $\Qq$-factorialization and let $K_{Y_i}+B_{Y_i}:=f_i^*(K_{X_i}+B_i)$. Let $B_{Y_i,j}:=f_{i,*}^{-1}B_{i,j}$ for each $i$. Then $(Y_i,B_{Y_i})$ is $\Qq$-factorial klt near $f^{-1}(x_i)$ and $\mld(Y_i/X_i\ni x_i,B_{Y_i})=a_i\geq a$. By Lemma \ref{lem: zhu24 2.10 real}, $\nvol(Y_i\ni y_i,B_{Y_i})\geq\epsilon$ for any closed point $y_i\in f^{-1}(x_i)$. 

    Let $p_{\alpha}(B_{Y_i}):=\sum_{j=1}^{m_i}p_{\alpha}(b_{i,j})B_{Y_i,j}$. Since $Y_i$ is of Fano type over $X_i$, possibly replacing $Y_i$ with a minimal model of $K_{Y_i}+p_{\alpha}(B_{Y_i})$, we may assume that $K_{Y_i}+p_{\alpha}(B_{Y_i})$ is semi-ample$/X_i$. By Proposition \ref{prop: alct uniform increasing}, for any (not necessarily closed) point $\eta_i\in f_{i}^{-1}(x_i)$, $(Y_i\ni\eta_i,p_{\alpha}(B_{Y_i}))$ is $a$-lc. Thus $(Y_i/X_i\ni x_i,p_{\alpha}(B_{Y_i}))$ is $a$-lc.
    
    Let $E_i$ be a prime divisor over $X_i\ni x_i$ such that $a(E_i,X_i,B_i)=a_i$ and let $W_i:=\Center_{Y_i}E_i$. Let $\eta_{W_i}$ be the generic point of $W_i$. Since $\Ii\subset [0,1]$ is a DCC set and $Y_i$ is $\Qq$-factorial near $W_i$, by \cite[18.22 Theorem]{Kol+92}, possibly reordering indices and passing to a subsequence, we may assume that there exists an integer $m\leq m_i$, such that $B_{Y_i,j}$ passes through $\eta_{W_i}$ if and only if $1\leq j\leq m$. We let $\bar b_j:=\lim_{i\rightarrow+\infty}b_{i,j}$ for any $1\leq j\leq m$. Possibly passing to a subsequence, we have $\bar b_j\leq p_{\alpha}(b_{i,j})$ for any $i$ and $1\leq j\leq m$.
    
    By our construction, $\mld(Y_i\ni\eta_{W_i},p_{\alpha}(B_{Y_i}))\in [a,a_i]$. By Lemma \ref{lem: xz21 1.4 nonclosed point}, there exits a positive integer $I$ depending only on $d$ and $\epsilon$, such that $ID$ is Cartier near $W_i$ for any Weil divisor $D$ on $Y_i$. By \cite[Theorem 1.2]{Nak16}, $\mld(Y_i\ni\eta_{W_i},p_{\alpha}(B_{Y_i}))$ belongs to a discrete set. Thus possibly passing to a subsequence, we may assume that $\mld(Y_i\ni\eta_{W_i},p_{\alpha}(B_{Y_i}))=a$ for each $i$. Moreover, possibly passing to a subsequence, we have
    $$a=\mld(Y_i\ni\eta_{W_i},p_{\alpha}(B_{Y_i}))\leq \mld\left(Y_i\ni\eta_{W_i},\sum_{j=1}^m\bar b_jB_{Y_i,j}\right)\leq\mld(Y_i\ni\eta_{W_i},B_{Y_i})=a_i.$$
     By \cite[Theorem 1.2]{Nak16} again, possibly passing to a subsequence, we have $$\mld\left(Y_i\ni\eta_{W_i},\sum_{j=1}^m\bar b_jB_{Y_i,j}\right)=a,$$
     thus $\bar b_j=p_{\alpha}(b_{i,j})$ for any $i$ and $1\leq j\leq m$. In particular, $\lim_{i\rightarrow+\infty}||p_{\alpha}(b_{i,j})-b_{i,j}||=0$ for any $1\leq j\leq m$.

     We run a $-(K_{Y_i}+p_{\alpha}(B_{Y_i}))$-MMP$/X_i$ which terminates with a model $V_i$ such that $-(K_{V_i}+p_{\alpha}(B_{V_i}))$ is semi-ample$/X_i$, where $p_{\alpha}(B_{V_i})$ is the image of $p_{\alpha}(B_{Y_i})$ on $V_i$. Let $B_{V_i},B_{V_i,j}$ be the images of $B_{Y_i},B_{Y_i,j}$ on $V_i$ respectively. Then $\mld(V_i/X_i\ni x_i,B_{V_i})=a_i$, so by construction of $p_{\alpha}$, 
     $$\mld(V_i/X_i\ni x_i,p_{\alpha}(B_{V_i}))\geq a.$$
     Since $Y_i\dashrightarrow V_i$ is a $-(K_{Y_i}+p_{\alpha}(B_{Y_i}))$-MMP$/X_i$, we have
     $$a=\mld(Y_i/X_i\ni x_i,p_{\alpha}(B_{Y_i}))\geq\mld(V_i/X_i\ni x_i,p_{\alpha}(B_{V_i})).$$
    So
     $$\mld(V_i/X_i\ni x_i,p_{\alpha}(B_{V_i}))=a.$$
     Thus $(V_i/X_i\ni x_i,p_{\alpha}(B_{V_i}))$ is $(a,\mathbb R)$-complementary. By \cite[Lemma 2.11]{CH21}, $(Y_i/X_i\ni x_i,p_{\alpha}(B_{Y_i}))$ is $(a,\mathbb R)$-complementary. Let $(Y_i/X_i\ni x_i,p_{\alpha}(B_{Y_i})+G_{Y_i})$ be an $(a,\mathbb R)$-complement of $(Y_i/X_i\ni x_i,p_{\alpha}(B_{Y_i}))$. Let $Z_i$ be the ample model$/X_i$ of $K_{Y_i}+p_{\alpha}(B_{Y_i})$ and let $p_{\alpha}(B_{Z_i}),B_{Z_i},G_{Z_i}$ be the images of $p_{\alpha}(B_{Y_i}),B_{Y_i},G_{Y_i}$ on $Z_i$ respectively. Then $-G_{Z_i}$ is ample$/X_i$. 
     % Since $-(K_{V_i}+p_{\alpha}(B_{V_i}))$ is big and nef$/X_i$,

  Let $g_i: Z_i\rightarrow X_i$ be the induced birational morphism. If $g_i$ is not the identity morphism, then $G_{Z_i}\neq 0$ and $\Supp G_{Z_i}$ contains $g_i^{-1}(x_i)$, we have
  \begin{align*}
a=&\mld(Y_i/X_i\ni x_i,p_{\alpha}(B_{Y_i}))=\mld(Z_i/X_i\ni x_i,p_{\alpha}(B_{Z_i}))\\
>&\mld(Z_i/X_i\ni x_i,p_{\alpha}(B_{Z_i})+G_{Z_i})=a,
\end{align*}
a contradiction. Therefore, $g_i$ is the identity morphism, so $K_{X_i}+p_{\alpha}(B_i)$ is $\Rr$-Cartier, and
$$K_{Y_i}+p_{\alpha}(B_{Y_i})=f_i^*(K_{X_i}+p_{\alpha}(B_i)).$$
Thus $\mld(X_i\ni x_i,p_{\alpha}(B_i))=\mld(Y_i/X_i\ni x_i,p_{\alpha}(B_{Y_i}))=a$. This implies (5).

Finally, we prove (6). Suppose that (6) does not hold, then possibly reordering indices and passing to a subsequence, we may assume that there exists a positive integer $n$ such that $p_{\alpha}(b_{i,m+1})-b_{i,m+1}>\frac{\alpha}{n}$ for each $i$. Then by (5), possibly passing to a subsequence, we have
$$\mld\left(X_i\ni x_i,p_{\frac{\alpha}{n}}(B_i)\right)=a.$$
%:=\sum_{i=1}^{m_i}p_{\frac{\alpha}{n}}(b_{i,j})B_{i,j}
By (3), $p_{\alpha}(B_i)\geq p_{\frac{\alpha}{n}}(B_i)$. Moreover, since $p_{\alpha}(b_{i,m+1})-b_{i,m+1}>\frac{\alpha}{n}$, $p_{\alpha}(B_i)-p_{\frac{\alpha}{n}}(B_i)\neq 0$, and 
$$a=\mld\left(X_i\ni x_i,p_{\frac{\alpha}{n}}(B_i)\right)>\mld\left(X_i\ni x_i,p_{\alpha}(B_i)\right)=a,$$
a contradiction. 
\end{proof}

%Let $F_i$ be a prime divisor over $X_i\ni x_i$ such that $a(F_i,X_i,p_{\frac{\alpha}{n}}(B_i))=a$, then since $x_i\in B_{i,j}$ for any $i,j$, we have $\mult_{F_i}(p_{\alpha}(B_i)-p_{\frac{\alpha}{n}}(B_i))>0$. Thus
%$$a=a(F_i,X_i,p_{\frac{\alpha}{n}}(B_i))>a(F_i,X_i,p_{\alpha}(B_i))\geq a,$$a contradiction.

\section{ACC for local volumes}

\begin{prop}\label{prop: zhu24 4.8 generalized}
Let $d$ be a positive integer and $\epsilon$ a positive real number. Then there exist two positive real numbers $\epsilon_0,\delta$ and a positive integer $N$ depending only on $d$ and $\epsilon$ satisfying the following. Assume that $(X\ni x,B)$ is a germ of dimension $d$ such that the coefficients of $B$ are $\geq\epsilon$ and $\nvol(X\ni x,B)\geq\epsilon$. Then there exists a $\Qq$-divisor $B^+$, such that
\begin{enumerate}
    \item $NB^+$ is integral,
    \item $B^+\geq (1+\delta)B$, and
    \item $\nvol(X\ni x,B^+)\geq\epsilon_0$.
\end{enumerate}
\end{prop}

\begin{proof}
    It essentially follows from \cite[Proof of Theorem 4.8]{Zhu24}. 

Let $c:=c(d)$ where $c(d)$ is the positive real number defined in Theorem \ref{thm: zhu24 1.3 real}.
    
    Let $N_1$ be a positive integer depending only on $c$ and $\epsilon$ such that $\lfloor (1+\frac{c}{2})N_1\beta\rfloor\geq (1+\frac{c}{3})N_1\beta$ for any real number $\beta\in [\epsilon,1]$. Let $\bar B:=\frac{1}{N_1}\lfloor (1+\frac{c}{2})N_1B\rfloor$. Then $\bar B\geq (1+\frac{c}{3})B$.

    Possibly shrinking $X$ near $x$, we may assume that $X$ is affine. Let $f: X'\rightarrow X$ be a small $\Qq$-factorialization, $K_{X'}+B':=f^*(K_X+B)$, and $\bar B':=\frac{1}{N_1}\lfloor (1+\frac{c}{2})N_1 B'\rfloor$. Since $X'$ is of Fano type over $X$, we may run a $-(K_{X'}+\bar B')$-MMP$/X$ which terminates with a good minimal model $X''$. Possibly replacing $X'$ with $X''$, we may assume that $-(K_{X'}+\bar B')$ is nef$/X$. 

    By Lemma \ref{lem: zhu24 2.10 real}, for any closed point $x'\in f^{-1}(x)$,
    $$\nvol(X'\ni x',B')\geq\nvol(X\ni x,B)\geq\epsilon.$$
    By Theorem \ref{thm: zhu24 1.3 real} and our choice of $c$, there exists a positive real number $c$ depending only on $d$ and $\epsilon$ such that $(X'\ni x',(1+c)B')$ is lc for any closed point $x'\in f^{-1}(x)$, so $(X'/X\ni x,(1+c)B')$ is lc. By Lemma \ref{lem: zhu24 2.11 real}, there exists a positive real number $\epsilon_1$ depending only on $d$ and $\epsilon$ such that $\nvol(X'\ni x',(1+\frac{c}{2})B')\geq\epsilon_1$ for any closed point $x'\in f^{-1}(x)$. Since 
    $$\left(1+\frac{c}{2}\right)B'\geq\bar B'\geq B',$$
    $\nvol(X'\ni x',\bar B')\geq\nvol(X'\ni x',(1+\frac{c}{2})B')\geq\epsilon_1$ for any closed point $x'\in f^{-1}(x)$. Since $(X'/X\ni x,(1+c)B')$ is lc and $(1+\frac{c}{2})B'\geq\bar B'$, there exists a positive real number $\gamma_1<1$ depending only on $c$ such that $(X'/X\ni x,(1+\gamma_1)\bar B')$ is klt. In particular, $(X'/X\ni x,\bar B')$ is klt.

   By our construction $N_1(K_{X'}+\bar B')$ is a Weil divisor. By Lemma \ref{lem: xz21 1.4 nonclosed point} there exists a positive integer $I$ depending only on $d$ and $\epsilon$ such that $IN_1(K_{X'}+\bar B')$ is Cartier. Since $-(K_{X'}+\bar B')$ is big$/X$ and nef$/X$, and $(X'/X\ni x,\bar B')$ is klt, by the relative version of Koll\'ar's effective base-point-free theorem \cite[Theorem 1.3]{Fuj09}, there exists an integer $N\geq 3$ depending only on $d$ and $\epsilon$ such that $-N(K_{X'}+\bar B')$ is base-point-free$/X$. Let $D'$ be a general element of $|-N(K_{X'}+\bar B')|$. Since $(X'/X\ni x,(1+\gamma_1)\bar B')$ is klt and $\gamma_1<1$, $(X'/X\ni x,(1+\gamma_1)(\bar B'+\frac{1}{N}D'))$ is klt.

   Let $D:=f_*D'$, then 
   $$K_{X'}+\bar B'+\frac{1}{N}D'=f^*\left(K_X+\bar B+\frac{1}{N}D\right).$$
   In particular, $K_X+\bar B+\frac{1}{N}D$ is $\Qq$-Cartier, so $\bar B-B+\frac{1}{N}D$ is $\Rr$-Cartier. Since $f$ is small,
   $$K_{X'}+(1+\gamma_1)\left(\bar B'+\frac{1}{N}D\right)-\gamma_1\bar B=f^*\left(K_X+B+(1+\gamma_1)\left(\bar B-B+\frac{1}{N}D\right)\right),$$
   so $(X\ni x,B+(1+\gamma_1)(\bar B-B+\frac{1}{N}D))$ is klt. By Lemma \ref{lem: zhu24 2.11 real}, there exists a positive real number $t$ depending only on $d$ and $\gamma_1$, such that
   $$\nvol\left(X\ni x,\bar B+\frac{1}{N}D\right)\geq t\cdot\nvol(X\ni x,\bar B)\geq t\epsilon.$$
   Note that $t$ only depends on $d$ and $\epsilon$. We may let $B^+:=\bar B+\frac{1}{N}D$, $\epsilon_0:=t\epsilon$, and $\delta:=\frac{c}{3}$ as required.
\end{proof}

\begin{thm}\label{thm: acc nv with mld upper bound}
Let $d$ be a positive integer, $l$ a positive real number, and $\Ii\subset [0,1]$ a DCC set. Then
$$\left\{\nvol(X\ni x,B)\Bigm|\dim X=d,B\in\Ii,\mld(X\ni x,B)\leq l, x\text{ is a closed point}\right\}$$
satisfies the ACC.
\end{thm}
\begin{proof}
Possibly replacing $\Ii$ with $\bar\Ii$, we may assume that $\Ii=\bar\Ii$. We may assume that $\Ii\cap (0,1]\not=\emptyset$, otherwise the proposition follows from \cite[Theorem 1.2]{XZ24}. Let $\gamma_0:=\min\Ii_{>0}$.

Suppose that the theorem does not hold. Then there exists a sequence of klt germs $(X_i\ni x_i,B_i)$ of dimension $d$ such that $B_i\in\Ii$, $\mld(X_i\ni x_i,B_i)\leq l$, and $v_i:=\nvol(X_i\ni x_i,B_i)$ is strictly increasing. In particular, we may assume that $v_i\geq\epsilon$ for some positive real number $\epsilon\in (0,\gamma_0]$.  By Proposition \ref{prop: zhu24 4.8 generalized}, there exist two positive real numbers $\epsilon_0,\delta$ and a positive integer $N$ depending only on $d,\epsilon$ and $\Ii$, such that there exists a $\Qq$-divisor $B_i^+$ on $X_i$ such that $NB_i^+$ is integral, $B_i^+\geq (1+\delta)B_i$, and $\nvol(X_i\ni x_i,B_i^+)\geq\epsilon_0$. 

%there exists a positive real number $a$ such that $a=\lim_{i\rightarrow+\infty}\mld(X_i\ni x_i,B_i)$, and 
Let $B_i=\sum_{j=1}^{m_i}b_{i,j}B_{i,j}$ where $B_{i,j}$ are the irreducible components of $B_i$ and $b_{i,j}\in\Ii$. Possibly replacing $B_i$ we may assume that $x_i\in B_{i,j}$ for any $i,j$. By Proposition \ref{prop: inversion of stability nv sequence}, possibly passing to a sequence, there exist a finite set $\Ii_0\subset\Ii$ and a projection $p:\Ii\rightarrow\Ii_0$, such that
\begin{itemize}
    \item $p(\gamma)\geq\gamma$ for any $\gamma\in\Ii$,
    \item $p(\gamma')\geq p(\gamma)$ for any $\gamma',\gamma\in\Ii$ such that $\gamma'\geq\gamma$,
    \item $K_{X_i}+p(B_i)$ is $\Rr$-Cartier, and
    \item $\lim_{i\rightarrow+\infty}||p(B_i)-B_i||=0$.
\end{itemize}
% $\mld(X_i\ni x_i,p(B_i):=\sum_{j=1}^{m_i}p(b_{i,j})B_{i,j})=a$,In particular, $K_{X_i}+p(B_i)$ is $\Rr$-Cartier,
Let $D_i:=p(B_i)-B_i$ and $t_i:=\lct(X_i\ni x_i,B_i;D_i)$. Since $(X_i\ni x_i,B_i^+)$ is lc, $\lim_{i\rightarrow+\infty}||p(B_i)-B_i||=0$, $B_i^+\geq (1+\delta)B_i$, and $B_i\in\Ii$, we have $t_i\to +\infty$. By Lemma \ref{lem: zhu24 2.11 real}, 
$$v_i\geq \nvol(X_i\ni x_i,p(B_i))\geq\left(\frac{t_i}{1+t_i}\right)^dv_i,$$
so possibly passing to a subsequence, we may assume that $\nvol(X_i\ni x_i,p(B_i))$ is strictly increasing. This contradicts \cite[Theorem 1.2]{XZ24}.
\end{proof}
%$\lim_{i\rightarrow+\infty}t_i=+\infty$

\section{Proof of the main theorems}

\begin{proof}[Proof of Theorem \ref{thm: ACC local volume}]
Let $\epsilon$ be a positive real number and let
$$\Vol^{\loc}_{d,\Ii,\epsilon}:=\Vol^{\loc}_{d,\Ii}\cap [\epsilon,+\infty).$$
By Corollary \ref{cor: nv lower bound imply mld}, there exists a positive real number $l$ depending only on $d$ and $\epsilon$, such that
$$\Vol^{\loc}_{d,\Ii,\epsilon}=\{\nvol(X\ni x,B)\mid\dim X=d,B\in\Ii,\mld(X\ni x,B)\leq l, x\text{ is a closed point}\}.$$
By Theorem \ref{thm: acc nv with mld upper bound}, $\Vol^{\loc}_{d,\Ii,\epsilon}$ is an ACC set. Thus $\Vol^{\loc}_{d,\Ii}$ is an ACC set.
\end{proof}

\begin{proof}[Proof of Theorem \ref{thm: acc mld nv lower bound no l}]
When $x$ is a closed point, then the theorem follows from Theorem \ref{thm: acc mld nv lower bound} and Corollary \ref{cor: nv lower bound imply mld}. In general, the theorem follows from the case when $x$ is a closed point and \cite[Proposition 2.3]{Amb99}. 
\end{proof}

\begin{proof}[Proof of Theorem \ref{thm: inversion stability local volume}]
Suppose that the theorem does not hold, then there exist $X_i\ni x_i,B_i,B_i',\tau_i$ corresponding to $X\ni x,B,B',\tau$ as in the assumptions, such that
\begin{itemize}
    \item $\lim_{i\rightarrow+\infty}\tau_i=0$, and
    \item either $K_{X_i}+B_i$ is not $\Rr$-Cartier near $x_i$ or $(X_i\ni x_i,B_i)$ is not klt near $x_i$.
\end{itemize}
%We claim that the coefficients of $B_i'$ belong to a DCC set $\Ii$. Suppose not, then there exists a strictly increasing sequence $\{b_{j,i_j}\}_{j=1}^{+\infty}$ such that $b_{j,i_j}$ is the coefficient of a component $S_j$ of $B_{i_j}'$. Since $\Ii_0$ is a finite set, possibly passing to a subsequence, we may assume that there exists a real number $\gamma_0\in\Ii_0$ such that $\gamma_0$ is the coefficient of $S_{j}$ in $B_{i_j}$, Possibly passing to a subsequence we may assume that $i_j$ is strictly increasing. Then $b_{j,i_j}\in (\gamma_0-\tau_{i_j},\gamma_0]$, which is not possible as $\lim_{i\rightarrow+\infty}\tau_{i_j}=0$.

Since for each $i$, the coefficients of $B_i'$ belong to the set 
$$\{ (\gamma_0-\tau_{i},\gamma_0] \mid \gamma_0\in \Ii_0\},$$
and $\lim_{i\rightarrow+\infty}\tau_i=0$, we conclude that all the coefficients of $B_i'$ belong to a DCC set $\Ii$.

Now we may write $B_i'=\sum_{j=1}^{m_i}b_{i,j}B_{i,j}$ where $B_{i,j}$ are the irreducible components of $B_i'$ and $b_{i,j}\in\Ii$. By Proposition \ref{prop: inversion of stability nv sequence} and Corollary \ref{cor: nv lower bound imply mld}, possibly passing to a subsequence, there exist a finite set $\Ii_0'\subset\bar\Ii$, and a projection $p:\Ii\rightarrow\Ii_0'$, such that
\begin{itemize}
    \item $p(\gamma)\geq\gamma$ for any $\gamma\in\Ii$,
    \item $(X_i\ni x_i,p(B_i'))$ is lc,
    \item $K_{X_i}+p(B_i')$ is $\Rr$-Cartier, and
    \item $\lim_{i\rightarrow+\infty}||p(B_i')-B_i'||=0$.
\end{itemize}
Since $\lim_{i\rightarrow+\infty}\tau_i=0$ and $\Ii_0,\Ii_0'$ are two finite sets, and
$$||p(B_i')-B_i||\le ||p(B_i')-B_i'||+||B_i'-B_i||,$$
possibly passing to a subsequence, we have $p(B_i')=B_i$ for each $i$. Therefore, $K_{X_i}+B_i$ is $\Rr$-Cartier. By Proposition \ref{prop: zhu24 4.8 generalized}, there exists a positive real number $\delta$ such that $(X_i\ni x_i,B_i^+)$ is klt and $B_i^+\geq (1+\delta)B_i'$ for some $B_i^+$. Possibly passing to a subsequence, $B_i^+\geq B_i$, so $(X_i\ni x_i,B_i)$ is klt near $x_i$, a contradiction.
\end{proof}

\begin{proof}[Proof of Theorem \ref{thm: upper bound mldk}]
Since $\Ii$ is a DCC set, $\Ii\backslash\{0\}$ has a positive lower bound. By Proposition \ref{prop: inversion of stability nv sequence}, there exist two positive real numbers $\delta,\epsilon_0$ and a positive integer $N$ depending only on $d,\Ii$ and $\epsilon$, and a $\Qq$-divisor $B^+$ on $X$, such that $NB^+$ is integral, $B^+\geq (1+\delta)B$, and $\nvol(X\ni x,B^+)\geq\epsilon_0$. By Proposition \ref{prop: inversion of stability nv sequence} again, there exist two positive real numbers $\delta',\epsilon_0'$, and a positive integer $N'$ depending only on $d,\Ii$ and $\epsilon$, and a $\Qq$-divisor $B'$ on $X$, such that $N'B'$ is integral, $B'\geq (1+\delta')B^+$, and $\nvol(X\ni x,B')\geq\epsilon_0'$. In particular, $(X\ni x,B')$ is lc, and there exists a positive real number $t$ depending only on $d,\Ii$ and $\epsilon$, such that $B'-B^+\geq t(B^+-B)$.

By \cite[Theorem 1.3]{XZ24}, there exists a positive real number $M$ depending only on $d,\Ii,\epsilon_0$, which in turn depends only on $d,\Ii$ and $\epsilon$, such that $\mld^K(X\ni x,B^+)\leq M$. Let $E$ be a Koll\'ar component of $(X\ni x,B^+)$ such that $a(E,X,B^+)\leq M$. Thus
\begin{align*}
   a(E,X,B)&=a(E,X,B^+)+\mult_E(B^+-B)\leq a(E,X,B^+)+\frac{1}{t}\mult_E(B'-B^+)\\
   &\leq M+\frac{1}{t}(a(E,X,B^+)-a(E,X,B'))\leq \left(1+\frac{1}{t}\right)M.
\end{align*}
Therefore, we may take $l:=\left(1+\frac{1}{t}\right)M$.
\end{proof}

\begin{proof}[Proof of Theorem \ref{thm: Un upper bound mld nv lower bound}]
By Corollary \ref{cor: nv lower bound imply mld}, there exists a positive real number $l'$ depending only on $d,\epsilon$, a prime divisor $E$ over $X$, $a(E,X,B)=\mld(X\ni x,B)\le l'$. By Theorem \ref{thm: zhu24 1.3 real}, there exists a positive real number $\delta$ depending only on $d,\epsilon$, such that $(X\ni x,(1+\delta)B)$ is lc. Therefore, 
$$\mult_{E} B=\frac{1}{\delta} \left(a(E,X,B)-a(E,X,(1+\delta)B)\right)\le \frac{l'}{\delta},$$
and
$$a(E,X,0)=a(E,X,B)+\mult_{E}B\le l'+\frac{l'}{\delta}.$$
We may let $l=l'+\frac{l'}{\delta}$ as required.
\end{proof}

The following lemma is a simple consequence of \cite[Theorem 1.8]{Bir21}.
\begin{lem}\label{lem: strictly lc complement}
Let $d$ be a positive integer, $\Ii_0\subset [0,1]\cap \Qq$ a finite set. Then exists a positive integer $N$ depending only on $d,\Ii_0$ satisfying the following. 

Let $(X,B)$ be a klt pair near a closed point $x\in X$, such that $\dim X=d$ and $B\in\Ii_0$. Then there exists a $\Qq$-divisor $B^{+}$ on $X$, such that
\begin{enumerate}
    \item $N(K_X+B^{+})$ is Cartier near $x$,
    \item $B^{+}>B$,
    \item $\mld(X\ni x, B^{+})=0$.
\end{enumerate}
In particular, $(X\ni x,B^{+})$ is lc.
\end{lem}

\begin{proof}
Let $f:Y\to X$ be a plt blow up of $(X\ni x,B)$, $E$ the corresponding Koll\'ar component, and $B_Y$ be strict transform of $B$ on $Y$. By \cite[Corollary 3.20]{HLS19} (a variation of \cite[Theorem 1.8]{Bir21}), there exist a positive integer $N$ depending only on $d,\Ii_0$, and a $\Qq$-divisor $B_{Y}^{+}$ on $Y$, such that  $B_{Y}^{+}\ge B_Y$, and $N(K_Y+B_{Y}^{+}+E)\sim 0$ and $(Y,B_{Y}^{+}+E)$ is lc over a neighborhood of $x$. Let $B^{+}=f_{*}B_Y^{+}$. Then $N(K_X+B^{+})$ is Cartier near $x$, and $\mld(X\ni x, B^{+})=0$ as $a(E,X,B^{+})=a(E,Y,B_{Y}^{+}+E)=0$. Since $\mld(X\ni x,B)>0$, we conclude that $B^{+}>B$.
\end{proof}

The following proof is inspired by a discussion with Chen Jiang.
\begin{proof}[Proof of Theorem \ref{thm: Birkar conje lct nv}] 
Possibly shrinking $X$ near $x$, we may assume that $X$ is affine. Let $f: X'\rightarrow X$ be a small $\Qq$-factorialization, $K_{X'}+B':=f^*(K_X+B)$. Since $X'$ is of Fano type over $X$, we may run a $-K_{X'}$-MMP$/X$ which terminates with a good minimal model $X''$. Possibly replacing $X'$ with $X''$, we may assume that $-K_{X'}$ is nef$/X$. By Lemma \ref{lem: zhu24 2.10 real}, for any closed point $x'\in f^{-1}(x)$,
    $$\nvol(X'\ni x')\ge \nvol(X'\ni x',B')\geq\nvol(X\ni x,B)\geq\epsilon.$$
  
  By Lemma \ref{lem: xz21 1.4 nonclosed point}, there exists a positive integer $I$ depending only on $d$ and $\epsilon$ such that $IK_{X'}$ is Cartier. Since $-K_{X'}$ is big$/X$ and nef$/X$, and $X'$ is klt over an open neighborhood of $x$, by the relative version of Koll\'ar's effective base-point-free theorem \cite[Theorem 1.3]{Fuj09}, there exists an integer $N\geq 3$ depending only on $d$ and $\epsilon$ such that $-NK_{X'}$ is base-point-free$/X$. Let $\tilde{D'}$ be a general element of $|-NK_{X'}|$, $\tilde{D}:=f_{*}\tilde{D'}$, and $D_1=\frac{1}{N}\tilde{D}$, then $(X,D_1)$ is klt near $x$, and $ND_1$ is integral, and $N(K_X+D_1)$ is Cartier near $x$. By Lemma \ref{lem: strictly lc complement}, possibly replacing $N$ with a multiple, there exists a $\Qq$-Cartier $\Qq$-divisor $D_2$, such that $(X\ni x,D_1+D_2)$ is lc, $\mld(X\ni x,D_1+D_2)=0$, and $N(K_X+D_1+D_2)$ is Cartier near $x$. In particular, $D\coloneqq ND_2>0$ is Cartier near $x$. 

%A \emph{real valuation} of $R$ is a map $v\colon {\rm Frac(R)}^\times\to \Gamma$ such that $v|_{\bC^*}=0$, $v(fg)=v(f)+v(g)$, and $v(f+g)\ge \min\{v(f),v(g)\}$.

By \cite[Theorem 1.6]{LX19}, there exists a Koll\'ar component $E$ of $X'\ni x'$ for some closed point $x'\in f^{-1}(x)$, such that $\nvol(E,X')\le d^d+1$. 
  By \cite[Lemma 3.4]{Zhu24} and \cite[Lemma 2.9]{LX19}(1), there exists a positive real number $c_0$ depending only on $d$, such that
\begin{align*}
    N\cdot \lct(X\ni x, B;D)=&\lct(X\ni x, B;D_2)\\
    \ge &c_0\cdot\frac{\epsilon}{\nvol(E,X,B)} \cdot \frac{a(E,X,B)}{\mult_{E}(D_2)}\\
    =&\frac{c_0\epsilon}{a^{d-1}(E,X,B)\vol_X(\ord_E)\mult_{E}(D_2)}\\
    \ge & \frac{c_0\epsilon}{a^{d-1}(E,X,B)\vol_X(\ord_E) a(E,X,D_1)}\\
    \ge &\frac{c_0\epsilon}{a^d(E,X')\vol_{X'}(\ord_E)}\\
    =&\frac{c_0\epsilon}{\nvol(E,X')}\ge \frac{c_0\epsilon}{d^d+1},
\end{align*}
%where the first inequality is \cite[Lemma 3.4]{Zhu24},
  where the second inequality follows from that $a(E,X,D_1+D_2)\ge 0$ as $(X\ni x,D_1+D_2)$ is lc. We may let $t=\frac{c_0\epsilon}{N(d^d+1)}$ as required.
\end{proof}

\section{Further discussions}\label{sec: further discusssion}

In this section, we provide a package to prove the ACC property for local invariants. The content of this section is inspired by discussions with Z. Zhuang.

%By \cite[Lemma 3.13]{HLS19}, if a klt germ $(X\ni x,B)$ admits a $\delta$-plt blow-up, then $(X\ni x,B)$ admits a $\Qq$-factorial weak $\delta$-plt blow-up.

\begin{defn}
    Let $d$ be a positive integer, and $\mathfrak{f}$ a non-negative function defined on lc germs of dimension $d$. We say that $\mathfrak{f}$ is \emph{discrete} if for any finite set $\Ii_0\subset [0,1]$, 
        $$\Ll_{d,\Ii_0}:=\left\{\mathfrak{f}(X\ni x,B) \Bigm| \dim X=d,B\in\Ii_0,x\text{ is a closed point}\right\}$$
        is a discrete set away from $0$, i.e. $\Ll_{d,\Ii_0}\cap [\delta,+\infty)$ is discrete for any $\delta>0$.

        We say that $\mathfrak{f}$ is \emph{bounded} if there exists a positive real number $l$, such that $\mathfrak{f}(X\ni x,B)\le l$ for any lc germ $(X\ni x,B)$ of dimension $d$.  

        We say that $\mathfrak{f}$ is \emph{strictly decreasing} if for any lc germs $(X\ni x,B)$, $(X\ni x,B')$  of dimension $d$ with $B'>B$ near $x$, we have $\mathfrak{f}(X\ni x,B)>\mathfrak{f}(X\ni x,B')$.
        % (i.e. $B'\geq B$ and $B'-B\not=0$)

        We say that $\mathfrak{f}$ is \emph{approximate} if there exists a function $C_d(t)$ defined on $\Rr$ such that $\lim_{t\to +\infty}C_d(t)=1$ and satisfies the following:  
        $$\mathfrak{f}(X\ni x,B+D)\ge C_d(t)\mathfrak{f}(X\ni x,B)$$
        for any lc germ $(X\ni x,B)$ and $\Rr$-Cartier $\Rr$-divisor $D>0$ with $\lct(X\ni x, B;D)\ge t$.  

        We say that $\mathfrak{f}$ is \emph{strictly $N$-complemented} if for any positive real number $\epsilon$, there exists a positive integer $N$ and a positive real number $\delta$ depending only on $d$ and $\epsilon$, such that for any lc germ $(X\ni x,B)$ of dimension $d$ such that the coefficients of $B$ are $\geq\epsilon$ and $\mathfrak{f}(X\ni x,B)\geq\epsilon$, then exists a $\Qq$-divisor $B^+$, such that $(X\ni x,B^{+})$ is an $N$-complement of $(X\ni x,B)$ with $B^+\geq (1+\delta)B$.

        We say that $\mathfrak{f}$ satisfies \emph{the inversion of stability of $\Rr$-Cartierness} if for any DCC set $\Ii\subset[0,1]$ with $\Ii=\overline{\Ii}$, and a sequence of lc germs $(X_i\ni x_i,B_i)$ of dimension $d$ with $\mathfrak{f}(X_i\ni x_i,B_i)\ge \epsilon$ for some positive real number $\epsilon$, 
        possibly passing to a sequence, there exist a positive real number $\epsilon_0$, a finite set $\Ii_0\subset\Ii$ and a projection $p:\Ii\rightarrow\Ii_0$, such that
\begin{itemize}
    \item $p(\gamma)\geq\gamma$ for any $\gamma\in\Ii$,
    \item $p(\gamma')\geq p(\gamma)$ for any $\gamma',\gamma\in\Ii$ such that $\gamma'\geq\gamma$,
    \item $K_{X_i}+p(B_i)$ is $\Rr$-Cartier,
    \item $\lim_{i\rightarrow+\infty}||p(B_i)-B_i||=0$, and
    \item $\mathfrak{f}(X_i\ni x_i,p(B_i))\ge \epsilon_0$.
\end{itemize}

\end{defn}

\begin{rem}
The definition of the inversion of stability of $\Rr$-Cartierness seems a bit complicated and artificial. If the number of irreducible components of $B_i$ are bounded from above by $m$ for any $i$, say $B_i=\sum_{j=1}^m b_{i,j}B_{i,j}$, then we only need to require that $K_{X_i}+\overline{B_i}$ is $\Rr$-Cartier and $\mathfrak{f}(X_i\ni x_i,\overline{B_i})\ge \epsilon_0$, where $b_j=\lim_{i\to +\infty}b_{i,j}$, and $\overline{B_i}=\sum_{j=1}^m b_{j}B_{i,j}$. However, if $X_i$ is not $\Qq$-factorial, then the number of irreducible components of $B_i$ could be unbounded, so we introduce the projection $p$ to overcome this difficulty. 
Indeed, if the ACC for $\mathfrak{f}$ (i.e. ACC for $\mathfrak{f}(X\ni x,B)$ when $\dim X=d$ and $B\in\Ii$) does not hold, then possibly passing to a subsequence, we have that $\mathfrak{f}(X_i\ni x_i,\overline{B_i})\ge \epsilon$ for some fixed $\epsilon>0$. Since we usually prove the ACC properties by contradiction, the condition $\mathfrak{f}(X_i\ni x_i,p(B_i))\ge \epsilon_0$ is natural in this sense.

By the proof of Proposition \ref{prop: inversion of stability nv sequence}, the inversion of stability for $\Rr$-Cartierness holds if (1) $\mathfrak{f}(Y\ni y,B_Y)\ge \mathfrak{f}(X\ni x,B)$, (2) the Cartier index of any $\Qq$-Cartier Weil divisor is bounded from above if $\mathfrak{f}(X\ni x,B)\ge \epsilon$, and (3) either 1) the ACC conjecture for mlds holds or 2) both the ACC conjecture for mlds and the ACC conjecture for $a$-lcts hold when $\mathfrak{f}(X\ni x,B)\ge \epsilon$.
%Moreover, since the inversion of stability for $Rr$-Cartierness is expected to be true, the last condition on $\mathfrak{f}$ in the theorem below might be redundant. 
\end{rem}

 %    We say that $\mathfrak{f}$ is \emph{$\delta$-plt blow-up} for some positive real number $\delta$ if there exists a positive real number $\epsilon$, such that for any lc germ $(X\ni x,B)$ with $\mathfrak{f}(X\ni x,B)\ge \epsilon$, $(X\ni x,B)$ is klt and admits a $\delta$-plt blow up.

%\han{add def of (1), (2), (3)?}

%If (1) $\mathfrak{f}(Y\ni y,B_Y)\ge \mathfrak{f}(X\ni x,B)$, (2) the Cartier index of any $\Qq$-Cartier Weil divisor is bounded from above if $\mathfrak{f}(X\ni x,B)\ge \epsilon$, and (3) the mld ACC holds for $\mathfrak{f}(X\ni x,B)\ge \epsilon$ (or $\mld(X\ni x,B)$ has an upper bound and $\mathfrak{f}$ is strictly $N$-complemented), then $\mathfrak{f}$ satisfies inversion of stability of $\Rr$-Cartierness.

%\han{maybe not a good notion \emph{inversion of stability of $\Rr$-Cartierness}}

\begin{thm}\label{thm: ACC criterion}
       Let $d$ be a positive integer, and $\mathfrak{f}$ a non-negative function defined on lc germs of dimension $d$. If $\mathfrak{f}$ is discrete, bounded, strictly decreasing, approximate, strictly $N$-complemented, and satisfies the inversion of stability of $\Rr$-Cartierness, then for any DCC set $\Ii\subset [0,1]$,
         $$\Ll_{d,\Ii}:=\left\{\mathfrak{f}(X\ni x,B) \Bigm| \dim X=d,B\in\Ii,x\text{ is a closed point}\right\}$$
       satisfies the ACC.
\end{thm}

\begin{proof}
Possibly replacing $\Ii$ with $\bar\Ii$, we may assume that $\Ii=\bar\Ii$. Since $\mathfrak{f}$ is discrete, we may assume that $\Ii\ne \{0\}$. Let $\gamma_0:=\min\Ii_{>0}$.

Suppose that the theorem does not hold. Then there exists a sequence of lc germs $(X_i\ni x_i,B_i)$ of dimension $d$ such that $B_i\in\Ii$, and $v_i:=\mathfrak{f}(X_i\ni x_i,B_i)$ is strictly increasing. In particular, we may assume that $f_i\geq\epsilon$ for some positive real number $\epsilon\in (0,\gamma_0]$. 

Let $B_i=\sum_{j=1}^{m_i}b_{i,j}B_{i,j}$ where $B_{i,j}$ are the irreducible components of $B_i$ and $b_{i,j}\in\Ii$. Possibly replacing $B_i$ we may assume that $x_i\in B_{i,j}$ for any $i,j$. Since $\mathfrak{f}$ satisfies the inversion of stability of $\Rr$-Cartierness, possibly passing to a sequence, there exist a finite set $\Ii_0\subset\Ii$ and a projection $p:\Ii\rightarrow\Ii_0$, such that
\begin{itemize}
    \item $p(\gamma)\geq\gamma$ for any $\gamma\in\Ii$,
    \item $p(\gamma')\geq p(\gamma)$ for any $\gamma',\gamma\in\Ii$ such that $\gamma'\geq\gamma$,
    \item $K_{X_i}+p(B_i)$ is $\Rr$-Cartier, and
    \item $\lim_{i\rightarrow+\infty}||p(B_i)-B_i||=0$.
\end{itemize}

Since $\mathfrak{f}$ is strictly $N$-complemented, there exist a positive integer $N$ and a positive real number $\delta$ depending only on $d,\epsilon$, and $\Qq$-divisors $B_i^+$, such that $(X_i,B_i^{+})$ is an $N$-complement of $(X_i,B_i)$ with $B_i^{+}\geq (1+\delta)B_i$.
% $\mld(X_i\ni x_i,p(B_i):=\sum_{j=1}^{m_i}p(b_{i,j})B_{i,j})=a$,In particular, $K_{X_i}+p(B_i)$ is $\Rr$-Cartier, $\lim_{i\rightarrow+\infty}||p(B_i)-B_i||=0$, $B_i^+\geq (1+\delta)B_i$, and $B_i\in\Ii$,

Let $D_i:=p(B_i)-B_i$ and $t_i:=\lct(X_i\ni x_i,B_i;D_i)$. Since $(X_i\ni x_i,B_i^+)$ is lc,  we have $t_i\to +\infty$.
% By Proposition \ref{prop: zhu24 4.8 generalized}, there exist two positive real numbers $\epsilon_0,\delta$ and a positive integer $N$ depending only on $d,\epsilon  $ and $\Ii$, such that there exists a $\Qq$-divisor $B_i^+$ on $X_i$ such that $NB_i^+$ is integral, $B_i^+\geq (1+\delta)B_i$, and $\nvol(X_i\ni x_i,B_i^+)\geq\epsilon_0$. 

%there exists a positive real number $a$ such that $a=\lim_{i\rightarrow+\infty}\mld(X_i\ni x_i,B_i)$, and 
%Let $B_i=\sum_{j=1}^{m}b_{i,j}B_{i,j}$, where $B_{i,j}$ are the irreducible components of $B_i$ and $b_{i,j}\in\Ii$. Possibly replacing $B_i$, we may assume that $x_i\in B_{i,j}$ for any $i,j$. Let $b_j:=\lim_{i\to+\infty} b_{i,j}$, and $\overline{B_i}:=\sum_{j=1}^{m}b_{j}B_{i,j}$. By Zhuang's lemma, $K_{X_i}+\overline{B_i}$ is $\Rr$-Cartier.  

Since $\mathfrak{f}$ is discrete and bounded, possibly passing to a subsequence, we may assume that $\mathfrak{f}(X_i\ni x_i, p(B_i))=v$ is a constant. Since $\mathfrak{f}$ is strictly decreasing, $v<v_i$ for any $i$.

Since $\mathfrak{f}$ is approximate,  
$$v=\mathfrak{f}(X_i\ni x_i,p(B_i))=\mathfrak{f}(X_i\ni x_i,B_i+D_i)\ge C_d(t_i)\mathfrak{f}(X_i\ni x_i,B_i)=C_d(t_i)v_i,$$ and 
$$v\ge \lim_{i\to +\infty}C_d(t_i)v_i=\lim_{i\to +\infty}v_i>v,$$
a contradiction. 
\end{proof}
%Let $D_i=p(B_i)-B_i$. By Proposition \ref{prop: zhu24 4.8 generalized}, $t_i:=\lct(X_i\ni x_i,B_i;D_i)\to +\infty$. 

%In the following, we show that if $\mathfrak{f}$ is $\delta$-plt blow-up, then $\mathfrak{f}$ is inversion of stability of $\Rr$-Cartierness and strictly $N$-complemented. 

%\begin{lem}  Let $d$ be a positive integer, and $\mathfrak{f}$ a non-negative function defined on lc germs of dimension $d$. If $\mathfrak{f}$ is $\delta$-plt blow-up, then $\mathfrak{f}$ is strictly $N$-complemented.\end{lem}\han{need mld bdd}\begin{proof}This follows from \cite[Theorem 1.5]{HLQ23} and Proposition \ref{prop: zhu24 4.8 generalized}. \end{proof}

In the following, we will give an alternative proof of Theorem \ref{thm: ACC local volume} by applying Theorem \ref{thm: ACC criterion}.

Let's recall the definition of $\delta$-plt blow-ups.
\begin{defn}\label{defn: reduced component}
	Let $\delta$ be a positive real number, and $(X,B)$ a klt pair near a closed point $x$. 
 
	A \emph{$\delta$-plt blow-up} of $(X\ni x,B)$ of $(X\ni x,B)$) is a blow-up $f: Y\rightarrow X$ with the exceptional divisor $E$ over $X\ni x$, such that $(Y,f^{-1}_*B+E)$ is $\delta$-plt near $E$, and $-E$ is ample over $X$. %The divisor $E$ is called \emph{a reduced component} of $(X\ni x,B)$, and we also call it \emph{the reduced component} of $f$. 

A $\Qq$-factorial weak \emph{$\delta$-plt blow-up} of $(X\ni x,B)$) is a blow-up $f: Y\rightarrow X$ with the exceptional divisor $E$, such that $(Y,f^{-1}_*B+E)$ is $\Qq$-factorial $\delta$-plt near $E$, $-E$ is nef over $X$, $-(K_Y+B_Y+E)|_{E}$ is big, and $f^{-1}(x)=\Supp E$. %The divisor $E$ is called \emph{the reduced component} of $f$. 
\end{defn}

Based on \cite{Zhu24}, \cite{XZ24}, we confirm \cite[Conjecture 1.6]{HLQ23}.
 \begin{thm}\label{thm:existence of delta-plt blow-up}
	Let $d\ge 2$ be a positive integer and $\epsilon$ a positive real number. Then there exists a positive real number $\delta$ depending only on $d,\epsilon$ satisfying the following. 
		If $(X\ni x,B)$ is a klt germ of dimension $d$ such that the coefficients of $B$ are $\ge \epsilon$, and $\nvol(X\ni x,B)\ge \epsilon$, then $(X\ni x,B)$ admits a $\delta$-plt blow-up.
	\end{thm}

 \begin{proof}
By Proposition \ref{prop: zhu24 4.8 generalized}, there exist a positive real number $\epsilon_0$, and a finite set $\Ii_0\subset \Qq$ depending only on $d,\epsilon$, and a $\Qq$-divisor $B^+$, such that $B^+\in\Ii_0$, $B^+\ge B$, and $\nvol(X\ni x,B^+)\geq\epsilon_0$. By \cite[Theorem 1.5]{Zhu24} and \cite[Theorem 1.3]{XZ24}, $(X\ni x,B^+)$ admits a $\delta$-plt blow-up, which is also a $\delta$-plt blow-up of $(X\ni x,B)$ as $B^+\ge B$.
 \end{proof}

%$\mathcal{C}:=\{X\}$ be a set of projective Fano type varieties of dimension $d$, If $\mathcal{C}$ belongs to $\mathcal{X}\to \mathcal{T}$,

The following lemma is a generalization of the Hodge index theorem. %, which should be well-known to the experts, so we omit the proof.
\begin{lem}\label{lem: D^2H^n-2 DH^n-1} Let $X$ be a normal projective variety of dimension $d$ and $D$ an $\Rr$-Cartier $\Rr$-divisor on $X$ such that $D\cdot H^{d-1}= 0$ and $D^2\cdot H^{d-2}= 0$ for some ample $\mathbb{R}$-Cartier $\Rr$-divisor $H$, then $D\equiv 0$.
\end{lem}

\begin{proof}
We proceed by dimension induction. When $\dim X=2$, let $f:Y\to X$ be the minimal resolution, and $E\geq 0$ an anti-ample$/X$ exceptional$/X$ divisor. Then there exists a positive real number $\epsilon\ll 1$, such that $H'=f^*H-\epsilon E$ is ample. We have $f^*D\cdot H'=f^{*}D\cdot (f^*H-\epsilon E)=0$, and $(f^*D)^2=D^2=0$. By the Hodge index theorem, we have $f^*D\equiv 0$, thus $D\equiv 0$.

When $\dim X\ge 3$, let $C$ be any curve on $X$, then there exist  $n\gg 0$ and a divisor $H_0\in |nH|$ such that $C\subset H_0$. We have $D|_{H_0}\cdot (H|_{H_0})^{d-2}=nD\cdot H^{d-1}=0$ and $(D|_{H_0})^2\cdot (H|_{H_0})^{d-3}=nD^2\cdot H^{d-2}=0$. By induction, $D|_{H_0}$ is numerically trivial, thus $D\cdot C=D|_{H_0}\cdot C=0$, which implies that $D\equiv 0$ on $X$.
\end{proof}

%\han{add proof}

The following lemma is suggested by Chen Jiang.

\begin{lem}\label{lem: family numeric trivial}
Let $d$ be a positive integer, $\mathcal{X}\to \mathcal{T}$ a projective family bewteen normal quasi-projective varieties, such that $\dim(\mathcal{X}/\mathcal{T})=d$ and the morphism $N^1
(\mathcal{X}/\mathcal{T})\to N^1(X_t)$ is surjective for any general closed point $t\in \mathcal{T}$. Then there exists finitely many ample$/\mathcal{T}$ Cartier divisors $\mathcal{H}_i$ on $\mathcal{X}$ satisfying the following. For any $\Rr$-Cartier divisor $D$ on $\mathcal{X}_t$ for some general closed point $t\in T$, if $D\cdot \mathcal{H}_i|_{\mathcal{X}_t}^{d-1}=0$, then $D\equiv 0$. 
\end{lem}

\begin{proof}
Note that $N^1
(\mathcal{X}/\mathcal{T})\to N^1(\mathcal{X}_t)$ is surjective, let $\mathcal{D}$ be a preimage of $D$. 
Let $\mathcal{H}_1',\mathcal{H}_2',\ldots, \mathcal{H}_k'$ be a basis of $N^1
(\mathcal{X}/\mathcal{T})$, and $H_j'=\mathcal{H}_j'|_{\mathcal{X}_t}$. Suppose that $\left(\mathcal{D}\cdot (\sum_{j=1}^k s_j\mathcal{H}_j')^{d-1}\right)|_{\mathcal{X}_t}=D\cdot (\sum_{j=1}^k s_jH_j')^{d-1}=0$
for any positive integers $1\le s_j\le d$, $1\le j\le k$. Then by the Combinatorial Nullstellensatz, $D\cdot \prod_{j=1} H_j^{d_j}=0$ for any non-negative integer $d_j$ with $\sum_{j=1}^k d_j=d-1$. In particular, $D\cdot H_1^{d-1}=0$. Since $H_j$ spans $N^1(\mathcal{X}_t)$, $D=\sum_{j=1}^k a_jH_j$ for some real number $a_j$, thus $D\cdot \sum_{j=1}^k a_jH_j\cdot H_1^{d-2}=D^2\cdot H_1^{d-2}=0$. We conclude that $D\equiv 0$ by Lemma \ref{lem: D^2H^n-2 DH^n-1}. We may let $\{\mathcal{H}_1,\mathcal{H}_2,\ldots, \mathcal{H}_{d^k}\}:=\{\sum_{j=1}^k s_j\mathcal{H}_j'\mid 1\le s_j\le d, 1\le j\le k\}$ as required.
\end{proof}

\begin{lem}\label{lem: Zhuang R-Cartier}
Let $d$ be a positive integer, and $\delta$ a positive real number. Then there exists an integer $m$ and finitely many rational affine spaces $\mathcal{L}_i\subseteq\Rr^m$, $1\le i\le s$ depending only on $d,\delta$ satisfying the following. 

If $(X\ni x,B)$ is a klt germ of dimension $d$ with $B\in \{0\}\cup [\delta,+\infty)$ which admits a $\delta$-plt blow-up, then 
\begin{enumerate}
    \item $\Supp B$ has at most $m$ irreducible components, and 
    \item if $\mld(X\ni x,B)\ge \delta$, then $\{(\tilde{b}_1,\tilde{b}_2,\ldots,\tilde{b}_m)\in\Rr^m \mid K_X+\sum_{i=1}^m \tilde{b}_iB_i \text{ is }\Rr\text{-Cartier}\}=\mathcal{L}_j$ for some $j$, where $B_i$ are distinct irreducible components of $B$.
\end{enumerate}
\end{lem}

\begin{proof}
Let $f':Y'\to X$ be a $\delta$-plt blow up of $(X\ni x,B)$ with the exceptional divisor $E'$. Let $Y\to Y'$ be a small $\Qq$-factorization and $f:Y\to Y'\to X$ the induced morphism. By \cite[Lemma 3.13]{HLS19}, $f:Y\to X$ is a $\Qq$-factorial weak $\delta$-plt blow-up of $(X\ni x,B)$. Let $K_E+B_E:=(K_Y+f^{-1}_*B+E)|_{E}$. Let $B=\sum_{i=1}^{m'} b_iB_i$ where $B_i$ are distinct irreducible components. Note that for each $b_j>0$, there exists an irreducible component $V$ of $B_E$ such that $\mult_{V} B_E=\frac{u-1+\sum_i n_ib_i}{u}<1$ for some positive integer $u$ and non-negative integer $n_i$ with $n_j\ge 1$. In particular, $\sum_{V}\mult_{V} B_E\ge \sum_{i=1}^{m'}b_i$, where the sum runs over all irreducible components $V$ of $B_E$

By \cite[Lemma 4.1]{HLS19}, $(E, B_E)$ is $\delta$-klt, and belongs to a log bounded family. Thus, we may pick a very ample Cartier divisor $H$ on $E$, such that $-K_E\cdot H^{d-2}$ is bounded from above by a positive real number $M$ which only depends on $d,\delta$. Moreover, $B_E\cdot H^{d-2}\le-K_E\cdot H^{d-2}\le M$. Hence 
$$m'\delta\le \sum_{i=1}^{m'}b_i \le B_E\cdot H^{d-2}\le M.$$
We may let $m=\lfloor\frac{M}{\delta}\rfloor$ as required.

(2) We may assume that $m=m'$, and $B=\sum_{i=1}^m b_iB_i$. Since $K_E+B_E=(K_Y+\sum_{i=1}^m \tilde{b}_iB_i+E)|_{E}\sim_{\Rr} aE|_{E}$, where $a=a(E,X,B)\ge \delta$, by \cite[Proposition 4.4]{HLS19}, there exists a positive integer $I$ depending only on $d,\delta$, such that $IE|_{E}$ is Cartier. Suppose that $E$ belongs to the bounded family $\mathcal{X}\to \mathcal{S}$, by the proof of \cite[Proposition 2.9]{HX15}, possibly passing to a subsequence and shrinking $\mathcal{S}$, we may assume that $\mathcal{S}$ is irreducible and the geometric generic fiber of $\mathcal{X}\to \mathcal{S}$ is of Fano type. By \cite[Proposition 2.8]{HX15}, there exists a finite dominant
morphism $\mathcal{T}\to\mathcal{S}$, such that
$\mathcal{X}_{\mathcal{T}}:=\mathcal{X}\times_{\mathcal{S}}\mathcal{T}$, $N^1
(\mathcal{X}_{\mathcal{T}}/\mathcal{T})\to N^1(\mathcal{X}_{\mathcal{T}_t})$ is an isomorphism for any point $t\in \mathcal{T}$. Let $\mathcal{H}_i$ be the ample Cartier divisors constructed in Lemma \ref{lem: family numeric trivial}. Let $H_i=\mathcal{H}_i|_{E}$ for each $i$.
Since $E$ is bounded, and $-IE|_{E}$ is Cartier and nef, possibly passing to a subsequence, we may assume that $-E|_{E}\cdot H_i^{d-2}$ is fixed for each $i$. 

Since $Y$ is of Fano type over $X$, by the cone theorem, $K_X+\tilde{B}:=K_X+\sum_{i=1}^m \tilde{b}_iB_i$ is $\Rr$-Cartier if and only if $K_Y+\sum_{i=1}^m f_{*}^{-1}\tilde{b}_iB_i+(1-\tilde{a})E\equiv 0/X$ for some real number $\tilde{a}$, the latter is equivalent to $K_E+\tilde{B}_E:=(K_Y+\sum_{i=1}^m f_{*}^{-1}\tilde{b}_iB_i+E)|_{E}\sim_{\Rr} \tilde{a}E|_{E}$. Thus by construction of $H_i$, $K_X+\tilde{B}$ is $\Rr$-Cartier if and only if $\frac{(K_E+\tilde{B}_E)\cdot H_i^{d-2}}{E|_{E}\cdot H_i^{d-2}}$ are all equal. Since $(E, B_E)$ is $\delta$-klt, and belongs to a log bounded family, possibly passing to a subsequence, for each irreducible component $V$ of $\tilde{B}_E$, $\mult_{V}\tilde{B}_E=\frac{u-1+\sum_i n_ib_i}{u}$ for some fixed positive integer $u$ and non-negative integer $n_i$, and $V\cdot H_i^{d-2}$ are fixed for each $i$ and $V$. Hence the equation $\frac{(K_E+\tilde{B}_E)\cdot H_1^{d-2}}{E|_{E}\cdot H_1^{d-2}}=\frac{(K_E+\tilde{B}_E)\cdot H_2^{d-2}}{E|_{E}\cdot H_2^{d-2}}=\cdots$ defines a rational affine space. We finish the proof.
\end{proof}
%Possibly replacing $E$ with its small $\Qq$-factorization, we may assume that $E$ is $\Qq$-factorial. By \cite[Theorem 1.10]{HLQ23}, the Cartier index of each irreducible component of $B_E$ is bounded from above by a positive integer $I$ which only depends on $d,\delta$, 

%In particular, there are only finitely many possibilities of $K_E\cdot H^{d-2}$. 

%\han{add delta plt imply nv}

\begin{proof}[Second proof of Theorem \ref{thm: ACC local volume}] 
We may assume $d\ge 2$. By Theorem \ref{thm: ACC criterion}, it suffices to verify that $\nvol$ is discrete, bounded, strictly decreasing, approximate, and strictly $N$-complemented, and satisfies the inversion of stability of $\Rr$-Cartierness.

By \cite[Theorem 1.2]{XZ24}, $\nvol$ is discrete. By \cite[Theorem 1.6]{LX19}, $\nvol$ is bounded. By definition, $\nvol$ is strictly decreasing. By Lemma \ref{lem: zhu24 2.11 real}, $\nvol$ is approximate. By Theorem \ref{thm:existence of delta-plt blow-up}, Lemma \ref{lem: Zhuang R-Cartier} and \cite[Lemma 5.17]{HLS19}, $\nvol$ is inversion of stability of $\Rr$-Cartierness. By Proposition \ref{prop: zhu24 4.8 generalized}, $\nvol$ is strictly $N$-complemented.
\end{proof}

\begin{rem}
Our first proof of Theorem \ref{thm: ACC local volume} relies on Corollary \ref{cor: nv lower bound imply mld}, that is, an upper bound for all $\mld(X\ni x,B)$. The proof provided in this section essentially relies on much deeper results: an upper bound for all $\mld^K(X\ni x,B)$, \cite[Theorem 1.5]{Zhu24}, and the BAB conjecture proved by Birkar.
\end{rem}

\begin{rem}
We may also apply Theorem \ref{thm: ACC criterion} to reprove the ACC for mlds of exceptional singularities \cite[Theorem 1.2]{HLS19}.
\end{rem}

\medskip

\noindent\textbf{Other behavior of numerical invariants.} In birational geometry, a lot of invariants for pairs with no boundary and pairs with fixed boundary coefficients are known, or conjectured to satisfy the ACC or DCC. Usually, the same property will hold for pair with DCC boundary coefficients. What we have proven in this paper confirms this for local volumes. 

Nevertheless, there is one more thing: local volumes for varieties with fixed coefficients are not only known to satisfy the ACC, but are also known to be discrete away from $0$, that is, their only accumulation point is $0$. This is very different from many other invariants (e.g. volume, mld). The following theorem gives a positive answer to a question asked by C. Xu and Z. Zhuang.%Z. Zhuang suggest us that one should not only consider the ``DCC coefficient" part but also should consider the ``ACC coefficient with a positive lower bound" part. More precisely, Z. Zhuang asked the following:

%\begin{ques}[Xu-Zhuang]\label{ques: acc coefficient}    Let $d$ be a positive integer, $\epsilon$ a positive real number, and $\Ii\subset [0,1]$ an ACC set such that $\Ii\backslash\{0\}\subset [\epsilon, 1]$. Is $0$ the only accumulation point of $\Vol^{\loc}_{d,\Ii}$ from above? i.e. does     $\Vol^{\loc}_{d,\Ii}\cap [\delta,+\infty)$     satisfy the DCC for any positive real number $\delta$? \end{ques}

Recall that $$\Vol_{d,\Ii}^{\loc}:=\left\{\nvol(X\ni x,B)\Bigm| \dim X=d,B\in\Ii,x\text{ is a closed point}\right\}.$$

\begin{thm}\label{thm: acc coefficient}
    Let $d$ be a positive integer, $\epsilon$ a positive real number, and $\Ii\subset [0,1]$ an ACC set such that $\Ii\backslash\{0\}\subset [\epsilon, 1]$. Then $0$ is the only accumulation point of $\Vol^{\loc}_{d,\Ii}$ from above, that is, $\Vol^{\loc}_{d,\Ii}\cap [\epsilon_0,+\infty)$
    satisfy the DCC for any positive real number $\epsilon_0$.
\end{thm}
\begin{proof}
Possibly replacing $\Ii$ with $\bar\Ii$, we may assume that $\Ii=\bar\Ii$. %We may assume that $\Ii\cap (0,1]\not=\emptyset$, otherwise the theorem follows from \cite[Theorem 1.2]{XZ24}. Let $\gamma_0:=\min\Ii_{>0}$.

%By Proposition \ref{prop: zhu24 4.8 generalized}, there exist two positive real numbers $\epsilon_0',\delta_0$ and a positive integer $N$ depending only on $d,\epsilon,\epsilon_0$ and $\Ii$, such that there exists a $\Qq$-divisor $B_i^+$ on $X_i$ such that $NB_i^+$ is integral, $B_i^+\geq (1+\delta_0)B_i$, and $\nvol(X_i\ni x_i,B_i^+)\geq\epsilon_0'$. 
Suppose that the theorem does not hold. Then there exists a sequence of klt germs $(X_i\ni x_i,B_i)$ of dimension $d$ such that $B_i\in\Ii$, and $v_i:=\nvol(X_i\ni x_i,B_i)\ge \epsilon_0$ is strictly decreasing. Let $v:=\lim_{i\to +\infty} v_i$. By Theorem \ref{thm:existence of delta-plt blow-up}, there exists a positive real number $\delta$ depending only on $d,\epsilon,\epsilon_0$, such that each $(X_i\ni x_i,B_i)$ admits a $\delta$-plt blow-up.  

Let $B_i=\sum_{j=1}^{m_i}b_{i,j}B_{i,j}$ where $B_{i,j}$ are the irreducible components of $B_i$ and $b_{i,j}\in\Ii$. Possibly replacing $B_i$ we may assume that $x_i\in B_{i,j}$ for any $i,j$. By Lemma \ref{lem: Zhuang R-Cartier}(1), possibly passing to a subsequence, we may assume that $m_i=m$ for any $i$, $\{b_{i,j}\}_{i=1}^{+\infty}$ is decreasing for each $j$. Let $b_j=\lim_{i\to +\infty}b_{i,j}$ for each $j$. By Lemma \ref{lem: closed point nv lower bound implies mld lower bound} and Lemma \ref{lem: Zhuang R-Cartier}(2), possibly passing to a subsequence, we may further assume that $K_{X_i}+\overline{B_i}$ is $\Rr$-Cartier for any $i$, where $\overline{B_i}=\sum_{j=1}^m b_jB_{i,j}$. By \cite[Theorem 1.2]{XZ24}, passing to a subsequence, we may assume that there exists a positive real number $\overline{v}>v$, such that $\nvol(X_i\ni x_i,\overline{B_i})=\overline{v}$ for any $i$.

% $\mld(X_i\ni x_i,p(B_i):=\sum_{j=1}^{m_i}p(b_{i,j})B_{i,j})=a$,In particular, $K_{X_i}+p(B_i)$ is $\Rr$-Cartier,
Let $D_i:=B_i-\overline{B_i}$ and $t_i:=\lct(X_i\ni x_i,\overline{B_i};D_i)$. Since $\lim_{i\rightarrow+\infty}||D_i||=0$, by Proposition \ref{prop: zhu24 4.8 generalized}, $t_i\to +\infty$. By Lemma \ref{lem: zhu24 2.11 real}, 
$$v_i=\nvol(X_i\ni x_i,B_i)\geq\left(\frac{t_i}{1+t_i}\right)^d\overline{v},$$
which implies that $v=\lim_{i\to+\infty} v_i\ge \overline{v}$, a contradiction. 
\end{proof}

%\begin{rem}    Of course, Zhuang's alternative approach also works for the proof of Question \ref{ques: acc coefficient}, as long as one has a non-closed point version of \cite[Theorem 1.3]{XZ24}.\end{rem}

\begin{rem}
It is worth mentioning that other invariants may not share similar properties as local volumes in Theorem \ref{thm: acc coefficient}. For example, consider
\[
    \MLD_{d,\Ii}^{\smooth}:=\{\mld(X\ni x,B)\mid \dim X=d, X\text{ is smooth}, B\in\Ii\}.
\]
Then, when $\Ii$ is a finite set, $\MLD_{d,\Ii}^{\smooth}$ is also a finite set, and in particular, it is discrete away from $0$ \cite{Kaw14}. When $\Ii$ is a DCC set, it is conjectured that $\MLD_{d,\Ii}^{\smooth}$ satisfies the ACC \cite{Sho88}. However, for $d=2$ and $\Ii=\{\frac{1}{2}+\frac{1}{n}\mid n\geq 3\}\subset [\frac{1}{2},+\infty)$, we have
$\mld(\mathbb \bA^2\ni o,(\frac{1}{2}+\frac{1}{n})(x^2+y^n=0))=0$ and $\mld(\mathbb \bA^2\ni o,\frac{1}{2}(x^2+y^n=0))=1$ for any $n$. Hence for any real number $a\in (0,1)$, there are two sequences $n_i,m_i$ such that $n_i>m_i$, and 
\[
    a_i:=\mld\left(\bA^2\ni o,B_i:=\left(\frac{1}{2}+\frac{1}{n_i}\right)(x^2+y^{m_i})\right)
\]
is strictly decreasing and $\lim_{i\rightarrow+\infty}a_i=a>0$. Thus, although $\Ii$ is ACC and has a positive lower bound $\frac{1}{2}$, $\MLD_{2,\Ii}^{\smooth}$ is not DCC away from $0$. In fact, the previous arguments show that $[0,1)\subset \MLD_{2,\Ii}^{\smooth}$.

Note that in the example above, we have $\lim_{i\rightarrow+\infty}\lct(\bA^2\ni o, B_i;B_i)=0$.%, so there is no constant $c>0$ such that $(\bA^2\ni o,(1+c)B_i)$ is lc.
\end{rem}

\end{document}